\documentclass[10pt]{amsart}

\usepackage{amsthm, amsfonts, amssymb, color}
 \usepackage{mathrsfs}
\usepackage{amsmath}
 \usepackage{amstext, amsxtra}
  \usepackage{txfonts}
 \usepackage[colorlinks, linkcolor=black, citecolor=blue, pagebackref, hypertexnames=false]{hyperref}

 \allowdisplaybreaks
\newtheorem{thm}{\hskip\parindent {Theorem}}[section]
\newtheorem{lem}{\hskip\parindent {Lemma}}[section]
\newtheorem{defi}{\hskip\parindent {Definition}}[section]
\newtheorem{cor}{\hskip\parindent {Corollary}}[section]
\newtheorem{rem}{\hspace\parindent{Remark}}[section]
\newtheorem{pro}{\hskip\parindent{Proposition}}[section]

\title[Endpoint Strichartz Estimates]{
Endpoint Strichartz Estimates for Charge Transfer Hamiltonians}
\author{Qingquan Deng, \ Avy Soffer \ and Xiaohua Yao}

\address {Qingquan Deng, Department of Mathematics,
 Central China Normal University, Wuhan, 430079, P.R. China}
\email{dengq@mail.ccnu.edu.cn}
\address{ Avy Soffer, Department of Mathematics, Rutgers University,
Piscataway, 08854-8019, USA}
\email{soffer@math.rutgers.edu}
\address{Xiaohua Yao, Department of Mathematics and  Hubei Province Key Laboratory of Mathematical Physics,
 Central China Normal University, Wuhan, 430079, P.R. China}
\email{yaoxiaohua@mail.ccnu.edu.cn }
\date{\today}

\begin{document}

\maketitle
\begin{abstract}
We prove the optimal Strichartz estimates for  Schr\"{o}dinger equations with charge transfer potentials and general source terms in $\mathbb{R}^n$ for $n\geq3$.  The proof is based on asymptotic completeness for the charge transfer models and the (weak) point-wise time decay estimates for the scattering states of such systems of Rodnianski, Schlag and Soffer \cite{RSS1}. The method extends for the matrix charge transfer problems.
\end{abstract}

 \tableofcontents

\section{Introduction}
\setcounter{equation}{0}
\subsection{Main results}

 For describing our main results, let us first recall the following definition of scalar charge transfer model.

%Along the way we prove other results, which may be of independent
%interest:
%If we let $P_{c}$ denote the projection on the subspace of scattering states,
%then we prove the weighted operators
%$
%\langle x\rangle^{-\sigma} P_{c}\  \langle x\rangle^{+\sigma}
%$ are bounded on $L^2$ for all real $\sigma$.

%We also prove AC and local decay  for the matrix charge transfer problem, together with $L^{p}$ estimates then imply Strichartz, see e.g. [JSS].

\begin{defi}\label{defi-1}By a {\rm (scalar) charge transfer model},  we mean the following time-dependent Schr\"{o}dinger equation
\begin{eqnarray}
\left\{%
\begin{array}{ll}\label{Sch eq}
i\partial_{t}\psi=-\frac{1}{2}\Delta\psi+\sum^{m}_{\kappa=1}V_{\kappa}(x-\vec{v}_{\kappa}t)\psi\\
\psi(s,\cdot)=\psi_{s}(x),\ \ x\in \mathbb{R}^{n}, \ s\in \mathbb{R}.\\
\end{array}%
\right.
\end{eqnarray}
where $\vec{v}_{j}$ are distinct vectors in $\mathbb{R}^{n}$ ($n\geq 3$) and for every $1\leq k\leq m$, the real potential $V_{\kappa}$ satisfies that
(i) $V_{\kappa}$ is exponentially localized smooth function of $\mathbb{R}^{n}$, and
(ii) 0 is neither a zero eigenvalue nor a zero resonance of Schr\"{o}dinger operator $H_{\kappa}=-\frac{1}{2}\Delta+V_{\kappa}(x).$
\end{defi}

Say that $\psi$ is a resonance if it is a distributional solution of equation $H_{\kappa}\psi=0$ and belongs to the space $L^{2}(\langle x\rangle^{-\sigma}dx)$ for some $\sigma>\frac{1}{2}$, but not for $\sigma=0$.  One should notice that the condition (i) here is assumed for convenience, not for optimality. The condition (ii) usually appeared in the studies and applications of general dispersive estimates, see e.g. \cite{JK, JSS2, JSS1, RSS1, Wei}. Moreover, it is known that there is no resonance for $n\geq5$.

Schr\"odinger equation (\ref{defi-1}) describes the $m$-centers of forces are traveling along the given straight line trajectories $\vec{v}_{\kappa}t \ (k=1, 2, \cdots, m)$ all of which act on a quantum mechanical particle of mass 1 through the potential $V_k(x)$. It is a well-known model has been thoroughly studied by many peoples, e.g. by \cite{Gr, Ya1, Wu, Zi} for the multi-channel scattering theory, and by Rodnianski, Schlag and Soffer \cite{RSS1, RSS2} for  time decay estimates. In this work, we are mainly interested in proving the homogeneous and inhomogeneous endpoint Strichartz estimates for
scalar (and matrix) charge transfer problem with initial data belonging to scattering states in $L^{2}(\mathbb{R}^{n})$ as $n\geq3$.
The method of proof relies on the use of the Kato-Jensen type estimates and AC (asymptotic completeness),
as well as on the known $L^1 \cap L^2$ into $L^{\infty}$ decay estimate (see Rodnianski, Schlag and Soffer \cite{RSS1, RSS2}) of the charge transfer Hamiltonian. The proof also requires new estimates of the (time-dependent) projection on the scattering states of $H(t)$ (see (\ref{th})). Moreover, the methods used in \cite{RSS1} required localization of initial data, which is not suitable for proving endpoint Strichartz estimates.

Without loss of generality, we shall assume that the number of potentials in Definition \ref{Sch eq} is $m=2$ and that the velocities are $\vec{v}_{1}=0$, $\vec{v}_{2}=(1,0,\ldots,0)=\vec{e}_{1}$. Thus let us turn to the problem
\begin{eqnarray}
\left\{%
\begin{array}{ll}\label{main-equation}
i\partial_{t}\psi=-\frac{1}{2}\Delta\psi+V_{1}\psi+V_{2}(x-\vec{e}_{1}t)\psi\\
\psi(s,\cdot)=\psi_{s}\\
\end{array}%
\right.
\end{eqnarray}
where $V_{1},\ V_{2}$  are exponentially localized smooth  potentials and satisfy the conditions of Definition \ref{Sch eq}. For studying the charge transfer model, the  Galilei transforms and their inverse will be often used
\begin{eqnarray}\label{Ga-1}
G_{\vec{v},y}(t)=e^{i\frac{\vec{v}^{2}}{2}t}e^{-ix\cdot \vec{v}}e^{i(y+t\vec{v})\vec{p}}, \ \ G_{\vec{v},y}(t)^{-1}=e^{-iy\cdot \vec{v}}G_{-\vec{v},-y}(t)
\end{eqnarray}
where $\vec{p}=-i\vec{\nabla}$ and $ y, \vec{v}\in \mathbb{R}^{n}$. When $y=0$, we simply write $G_{\vec{v}}(t)=G_{\vec{v},0}(t)$ and then $G_{\vec{v}}(t)^{-1}=G_{-\vec{v}}(t)$. Clearly, the Galilei transforms $G_{\vec{v},y}(t)$ are isometries on all $L^{p}(\mathbb{R}^n)$ ($1\le p\le \infty $) spaces. Moreover, one can easily check that
$
G_{\vec{v},y}(t)e^{it\frac{\Delta}{2}}=e^{it\frac{\Delta}{2}}G_{\vec{v},y}(0)
$
and the solution
\begin{eqnarray}\label{eq-2}
\psi(t,x):=G_{\vec{v},y}(t)^{-1}e^{-itH}G_{\vec{v},y}(0)\psi_{0},\ \ H=-\frac{1}{2}\Delta+V,
\end{eqnarray}
satisfies Scr\"odinger equation with a moving potential
\begin{eqnarray}\label{eq-3}
i\partial_{t}\psi=-\frac{1}{2}\Delta\psi+V(x-\vec{v}t-y)\psi,\ \ \psi(0,\cdot)=\psi_{0}\nonumber.
\end{eqnarray}
In the sequel, we will make use of these properties of $G_{\vec{v},y}(t)$ without further mentioning.

By the assumptions on potentials  $V_{1}$ and $\ V_{2}$ in Definition \ref{Sch eq},  it follows from the Birman-Schwinger principle (see e.g. Reed-Simon \cite {RS}) that Schr\"odinger operators $H_{1}$ and $H_{2}$ have only finite discrete negative eigenvalues with finite multiplicities.  So we can list $u_{1},\ldots,u_{m}$ and $w_{1},\ldots,w_{\ell}$ (counting multiplicity) be the normalized orthogonal bound states of $H_{1}$ and $H_{2}$ corresponding to these negative eigenvalues $\lambda_{1},\ldots,\lambda_{m}$ and $\mu_{1},\ldots,\mu_{\ell}$, respectively. Denote by $P_{b}(H_{1})$ and $P_{b}(H_{2})$ the projections onto the bound states space of $H_{1}$ and $H_{2}$, respectively and let $P_{ac}(H_{\kappa})=I-P_{b}(H_{\kappa})$, $\kappa=1,2$. The projections have the form
\begin{eqnarray}
P_{b}(H_{1})=\sum^{m}_{i=1}\langle \cdot,u_{i}\rangle u_{i},\ \ \ \ \ P_{b}(H_{2})=\sum^{\ell}_{j=1}\langle \cdot,w_{j}\rangle w_{j}.\nonumber
\end{eqnarray}
It is well-known that the solution would not disperse for arbitrary  initial data even for Schr\"{o}dinger equations with only one potential. We need to use appropriate projection to project away bound states for the problem (\ref{main-equation}). Thus the following orthogonality condition in the context of the charge transfer Hamiltonian (\ref{main-equation}) was introduced firstly by Rodnianski, Schlag and Soffer \cite{RSS1}.

\begin{defi}\label{defi-2} Let $U(t,s)$ be the two parameter unitary propagators of Schr\"odinger equation $(\ref{main-equation})$ and  $\psi(t,x)=U(t,s)f$ be the solution with an initial value $\psi(s)=f\in L^2(\mathbb{R}^n)$ (see e.g. Reed-Simon \cite{RS1}).  We say that $f$ \text{(}or also $\psi(t,\cdot)$\text{)} is asymptotically orthogonal to the bound states of $H_{1}$ and $H_{2}$ if
\begin{eqnarray}\label{eq-4}
\|P_{b}(H_{1},t)U(t,s)f\|_{L^2}+\|P_{b}(H_{2},t)U(t,s)f\|_{L^2}\rightarrow0\ as \ t\rightarrow +\infty,
\end{eqnarray}
where
\begin{eqnarray}\label{eq-5}
 P_{b}(H_{1},t)=P_{b}(H_{1}),\ \ P_{b}(H_{2},t)=G_{-\vec{e_{1}}}(t)P_{b}(H_{2})G_{\vec{e_{1}}}(t)
\end{eqnarray}
for all times $t$,  by assumptions that $\vec{v}_{1}=0$ and $\vec{v}_{2}=(1,0,\ldots,0)=\vec{e}_{1}$.
\end{defi}

For each $s\in \mathbb{R}$, denote by $\mathscr{A}(s)$ the set of $f$ satisfying (\ref{eq-4}).  We remark that the set $\mathscr{A}(s)$ is a closed subspace of $L^{2}$ and exactly coincides with the space of \emph{scattering states} beginning from $t=s$ for the charge transfer problem (see e.g. Theorem 1.1 of Graf \cite{Gr}). Let $P_c(s)$ denote the projections on the closed scattering subspace $\mathscr{A}(s)$ in $L^2$. If $V_2=0$, then clearly $P_c(s)\equiv P_{ac}(H_1)=1-P_{b}(H_1)$ (the absolute spectra projection of $H_1$). If $V_2\neq 0$, then $P_c(t)$ is very different from the
instantaneous projections $P_{ac}(H(t))$ on the continuous spectral part of the time-dependent Hamiltonian:
\begin{eqnarray}\label{th}
H(t)=-\frac{1}{2}\Delta+V_{1}+V_{2}(x-\vec{e}_{1}t).
\end{eqnarray}
Although the definition of $P_{ac}(H(t))$ is quite intuitive  and simple, however, in this paper we don't use them. Comparably, the projections $P_c(t)$ have several nice properties which play indispensable roles in our arguments. This will be shown in more detail in the next section.
Furthermore, if the orthogonality condition of Definition \ref{defi-2} on initial data $f$ is satisfied, then the following decay estimates for the charge transfer Hamiltonian have been established by using the multi-channels decomposition method in Rodnianski, Schlag and Soffer \cite{RSS1} and Cai \cite{Cai}.

\begin{pro}\label{ex-pro-1}
Let $U(t,s)$ denote the two parameter unitary propagators of Schr\"odinger equation $(\ref{main-equation})$. Then for any initial data $f\in L^{1}\cap\mathscr{A}(s)$, there exists some constant $C>0$ independent of $t,s$  such that the solution has the weak decay estimate
\begin{eqnarray}\label{weak estimate}
\|U(t,s)f\|_{L^{2}+L^{\infty}}\le C \langle t-s\rangle^{-n/2}\|f\|_{L^{1}\cap L^{2}},
\end{eqnarray}
holds, where the norms  $$\|f\|_{L^{2}+L^{\infty}}:=\inf_{f=h+g}(\|h\|_{L^2}+\|g\|_{L^\infty}), \  \|f\|_{L^{1}\cap L^{2}}:=\max(\|f\|_{L^1},\|f\|_{L^2})$$ with the dual relation $(L^2+L^\infty)^*=L^1\cap L^2$.

Moreover, if $\widehat{V_{1}},\widehat{V_{2}}\in L^{1}$ (the hat  ``  $\hat{}$ "   denotes Fourier transform), then one further has the strongly decay estimate
\begin{eqnarray}\label{ex-eq-1}
\|U(t,s)f\|_{L^{\infty}}\le C |t-s|^{-n/2}\|f\|_{L^{1}}.
\end{eqnarray}
\end{pro}

 We remark that the weak estimate (\ref{weak estimate}) was first proved in Rodnianski, Schlag and Soffer \cite{RSS1}, and the stronger $L^1\cap L^2\rightarrow L^\infty$ decay estimate
$
\|U(t,s)f\|_{L^{\infty}}\le C |t-s|^{-n/2}\|f\|_{L^{1}\cap L^{2}}
$
can also be found there. The $L^1\rightarrow L^\infty$ decay estimate (\ref{ex-eq-1}) was finally established in Cai \cite{Cai}. Notice that only the case $s=0$ was considered in \cite{RSS1} and \cite{Cai}, it is easy to check that the proof process in \cite{RSS1} can work well and have uniformly bounds for all initial value time $s$.

 Note that considering potentials that are moving and time-dependent, (Non-endpoint or endpoint) Strichartz estimates for charge transfer models can't be directly deduced by the $L^1-L^\infty$ decay estimate (\ref{ex-eq-1}) and the $T^{\ast}T$ duality method, as usually done in Journ\'{e}, Sogge and Soffer \cite{JSS1, JSS2}, Ginibe-Velo \cite {GV} and Kee-Tao \cite {KT} et al. In this paper we will prove the full Strichartz estimates for the (scalar and matrix) charge transfer models based on the  $L^1\cap L^2\rightarrow L^\infty$ decay estimates, not necessarily $L^{1}\rightarrow L^{\infty}$ estimates and the analysis of $P_c(s)$ as the projections on the scattering states.
Indeed, the use of $P_{c}(s)$ greatly simplifies the analysis of time dependent Hamiltonians. As we will show, this family of projection operators $P_{c}(s)$ intertwines in a nice way with the full dynamical operators $U(t,s)$:
\begin{eqnarray}\label{com-identity}
P_c(t)U(t,s)=U(t,s)P_c(s), \ t,\ s\in \mathbb{R}.
\end{eqnarray}
The crucial asymptotic completeness for the charge transfer models can be used to control uniformly the projections $P_c(s)$  in some weighted
spaces (see Proposition \ref{lem-1} of Section 2). This is then used together with Duhamel and reversed Duhamel
representation of the solution of the charge transfer problem, to prove local decay estimates for such dynamics (this was not proved or
used in Rodnianski, Schlag and Soffer \cite{RSS1, RSS2}), and then prove the endpoint Strichartz
estimates. The arguments above then allow us to prove the Strichartz estimates from a
$L^1\cap L^2\rightarrow L^\infty$ bounds and avoid the $T^{\ast}T$ duality method.
We mention that previously, the instantaneous projections $P_{ac}(H(t))$ on the continuous spectral part
of the time-dependent Hamiltonian $H(t)$  were used to study dispersive estimates. To get favorable commutation of these instantaneous
projections, one needs to modify the dynamics by extra commutator term, as
utilized by Perelman \cite{Per1}, see also \cite{Bam2, Bec3, CucMa1}. This commutator correction was done by following a
similar construction of Kato \cite{Kato1}, in his studies of the adiabatic theorem.

Now we state one of our main results:
\begin{thm}\label{thm1}
Let $U(t)\psi_0:=U(t,0)\psi_0$ be the solution of the equation (\ref{main-equation}) with an initial value $\psi_0\in \mathscr{A}(0)$ at $t=0$. Then one has the homogeneous Strichartz estimates
\begin{eqnarray}\label{Ho-S}
\|U(t)\psi_0\|_{L_{t}^{p}L_{x}^{q}}\lesssim \|\psi_{0}\|_{L^{2}}
\end{eqnarray}
where the admissible pair $(p,q)$ satisfies
\begin{eqnarray}\label{pair}
\frac{2}{p}=\frac{n}{2}-\frac{n}{q},\ \ \ \ \ 2\leq p\leq\infty, \ n\ge 3.
\end{eqnarray}
 An analogous statement holds for the equation (\ref{Sch eq}) with any finite number of moving potentials.
\end{thm}

Note that $P_c(t)U(t,0)f=U(t,0)P_c(0)f$ for any $f\in L^2$, the homogeneous Strichartz estimate (\ref{Ho-S}) is equivalent to the following operator form
\begin{eqnarray}\label{Ho-S'}
\|P_c(t)U(t)f\|_{L_{t}^{p}L_{x}^{q}}\lesssim \|f\|_{L^{2}}, \ \ f\in L^2,
\end{eqnarray}
which shows exactly that the scattering part of any solution $U(t)f$ satisfies  dispersive estimates. It actually happens even for equation with a source term $F(t, x)$.
Now consider the following nonhomogeneous charge transfer model:
\begin{eqnarray}
\left\{%
\begin{array}{ll}\label{main-equation-21}
i\partial_{t}\psi=-\frac{1}{2}\Delta\psi+V_{1}\psi+V_{2}(x-\vec{e}_{1}t)\psi+F(t,x)\\
\psi(0,\cdot)=\psi_{0}.\\
\end{array}%
\right.
\end{eqnarray}

\begin{thm}\label{thm2}
Let $\psi(t)$ be the solution of the equation (\ref{main-equation-21}) with any initial date $\psi_{0}\in L^2$ and  $(p,q)$ and $(\tilde{p},\tilde{q})$ are any two admissible pairs satisfying (\ref{pair}). If $F\in L^{\tilde{p}'}_{t}L^{\tilde{q}'}_{x} $, then the solution $\psi(t)$ satisfies the following nonhomogeneous Strichartz estimates
\begin{eqnarray}\label{NHS}
\|P_{c}(t)\psi(t)\|_{L^{p}_{t}L^{q}_{x}}\lesssim \|\psi_{0}\|_{L^{2}}+\|F\|_{L^{\tilde{p}'}_{t}L^{\tilde{q}'}_{x}}.
\end{eqnarray}
A similar statement also holds for the equation (\ref{main-equation-21}) with multi-moving potentials.
\end{thm}

\vskip0.2cm As we mentioned before, our methods also work well for matrix charge transfer models, which are related with the N-soliton dynamics and soliton with potential interaction problems on nonlinear Schr\"{o}dinger equations. Let us first introduce the definition of matrix charge transfer model.

\begin{defi}\label{defi-3}
By a {\rm matrix charge transfer model} we mean a system
\begin{eqnarray}
\left\{%
\begin{array}{ll}\label{MSch eq}
i\partial_{t}\vec{\psi}=\left(
  \begin{array}{cc}
    -\frac{1}{2}\Delta & 0 \\
    0 & \frac{1}{2}\Delta \\
  \end{array}
\right)\vec{\psi}+\sum^{m}_{\kappa=1}V_{\kappa}(t,x-\vec{v}_{\kappa}t)\vec{\psi}\\
\vec{\psi}(0,\cdot)=\vec{\psi}_{0},\ \ x\in \mathbb{R}^{n},\\
\end{array}
\right.
\end{eqnarray}
where $\vec{v}_{j}\ (j=1,2,\cdots, m)$ are distinct vectors in $\mathbb{R}^{n}$ ($n\geq 3$) and $V_{\kappa}\ (k=1,2,\cdots, m)$ are matrix potentials of form
\begin{eqnarray}
V_{\kappa}(t,x)=\left(
  \begin{array}{cc}
    U_{\kappa}(x) & -e^{i\theta_{\kappa}(t,x)}W_{\kappa}(x) \\
    e^{-i\theta_{\kappa}(t,x)}W_{\kappa}(x) & -U_{\kappa}(x) \\
  \end{array}
  \right)
\end{eqnarray}
where $U_{\kappa}$, $W_{\kappa}$  are some localized smooth functions and
\begin{eqnarray}\theta_{\kappa}(t,x)=(|\vec{v}_{\kappa}|^{2}+\alpha^{2}_{\kappa})t+2x\cdot \vec{v}_{\kappa}+\gamma_{\kappa},\ \alpha_{\kappa},\ \gamma_{\kappa}\in \mathbb{R},\ \alpha_{\kappa}\neq0.
\end{eqnarray}
\end{defi}

 Definition above may seem more complex compared to the scalar charge transfer model. In fact,  these systems arise in the study of stability of multi-soliton states or soliton-potential interaction of nonlinear Schr\"{o}dinger equations, see e.g. \cite{CucMa1, CuMa2, RSS2}. In \cite{RSS1},  Rodnianski, Schlag and Soffer also have established the $L^1\cap L^2\rightarrow L^\infty$ estimates for the matrix charge transfer model (\ref{MSch eq}) under the certain assumptions on the initial date $\vec{\psi}_{0}$ and the spectra of matrix type Schr\"odinger operators:
 \begin{eqnarray}\label{eq-34'}
\mathscr{H}_{\kappa}=\left(
  \begin{array}{cc}
    -\frac{1}{2}\Delta+\frac{1}{2}\alpha^{2}_{\kappa}+U_{\kappa} & -W_{\kappa} \\
    W_{\kappa} &  \frac{1}{2}\Delta-\frac{1}{2}\alpha^{2}_{\kappa}-U_{\kappa} \\
  \end{array}
  \right), \ \ \ \ 1\leq\kappa\leq m.
\end{eqnarray}
Hence we want to obtain endpoint Strichartz estimates for the matrix model by similar arguments as done for the scalar case.  We should remark that Strichartz estimates for matrix models proved here will be crucial for the further studies of N-soliton dynamics and soliton+potential perturbation problems. It allows us to study such problems with fairly minimal regularity estimates and assumptions. Another important aspect is that it simplifies the needed decay  estimates to control the modulation equations of the soliton parameters. Strichartz estimates may allow the removal of localization and smoothness of the initial data in problems of asymptotic stability of soliton dynamics.

In the introduction, we don't attempt to state the results on Strichartz estimates for the matrix model (\ref{MSch eq}) in order to avoid to introduce too many  notations Comparing with the scalar model (\ref{Sch eq}), we remark that the matrix propagator operator $\mathcal {U}(t,s)$ is not unitary on $L^2$ and $\mathscr{H}_{\kappa}$ is non-selfadjoint, which can lead to some additional difficulties than the scalar case.  We will address these  results in the whole section 3.

 \subsection{Further remarks}
Strichartz estimates for Schr\"odinger operator $-\Delta+V(x)$ have been widely studied by many people. It is hard to collect all related references. Let us mention some works on Strichartz estimates
based on various somehow different approaches, for instance, see Journ\'{e}, Soffer and Sogge \cite{JSS1, JSS2} (non-endpoint case) and Keel-Tao \cite {KT} (endpoint case) using $T^*T$ arguments by $L^1$-$L^\infty$ decay estimates, Yajima \cite{Ya2} by wave operator method,  Rodnianski and Schlag \cite{RoSc} using local decay estimates,  Beceanu \cite{Bec3} by  abstract Wiener Lemma,  Bouclet and Mizutani \cite{BM} recently by uniform Sobolev estimate and see Schlag's review paper \cite{Sch2}. Some methods also were extended to obtain Strichartz estimates for certain time-dependent potential Hamiltonians, see e.g. Goldberg \cite{Go}, Beceanu \cite{Bec3} and  Beceanu and Soffer \cite{BeSo1, BeSo2}. Of course, we remark that these results can not cover our case of more than one such moving potential.

For the matrix non-selfadjoint operator $\mathscr{H}$ as in (\ref{eq-34'}), there exists a great number of papers related to them, see e.g \cite{AS, AFS, Bam1, Bec2, BJ, Cuc1, CuMi, ES, FJGS, GNP, HM, MM1, MMT, Per2, Per1, Per5, Per4, RSS1, RSS2, Sch, SZ} and therein references, where all kinds of dispersive estimates for matric operator were discussed and used to study soliton scattering problems. As for Strichartz estimates, Schlag in \cite{Sch} got non-endpoint Strichartz estimates for a non-selfadjoint Hamiltonian and Beceanu \cite{Bec2} further obtained the endpoint one. Cuccagna and Mizumatchi \cite{CuMi} also established optimal Strichartz estimates for some matrix models based on the matrix wave operator methods from  Cuccagna \cite{Cuc1}, originally motivated by Yajima \cite{Ya2} for Schr\"odigner operator. More recently, the works \cite{Bec3, BeSo1, BeSo2} also deal with  some time-dependent matrix potentials, but none in all these studies above can obtain the endpoint Strictures estimates for matrix charge transfer model (\ref{MSch eq}). Actually, it is possible to use the idea of Beceanu \cite{Bec2} to get the endpoint Strictures estimates for matrix charge transfer model (\ref{MSch eq}). Precisely, one may obtain the desire estimates by showing
\begin{eqnarray}\label{ex-eq-4}
\|U(t)P_{c}(t)U^{\ast}(s)P^{\ast}_{c}(s)\|_{L^{1}\rightarrow L^{\infty}}\lesssim |t-s|^{-\frac{n}{2}}
\end{eqnarray}
and then applying Keel-Tao's argument. However, the proof for (\ref{ex-eq-4}) is highly non-trivial and may need much more complex computation than our method. Moreover, the wave operator argument as in \cite{Cuc1, CuMi} is also not an option for us since the construct for the inverse wave operator and the intertwining properties are all unknown for matrix charge transfer model (\ref{MSch eq}). Thus in this paper, we adopt a different approach which is based on decay estimate and AC (asymptotic completeness).

Finally, we mention the works of Cuccagna and Meada \cite{CucMa1, CuMa2} where weak Strichartz estimates are proved for a related charge transfer Hamiltonian model with stronger assumptions on the inhomogeneous term, The equations in \cite{CucMa1, CuMa2} are modified using idea of \cite{Bec3} to project on the scattering states of one fixed operator, the analysis is therefore more involved and different from our approach. It is not clear if the result of  \cite{CucMa1, CuMa2}  holds for all initial scattering states of the original charge transfer problem. For example in Theorem 7.1 of \cite{CuMa2}, their estimate holds only for initial data in the range of $P_{ac}(\mathscr{H}_{1})$.  Furthermore, they need the  much stronger bound on the inhomogeneous term
$
\big\|F\big\|_{L^{2}_{t}L^{2,\sigma}_{x}+L^{1}_{t}L^{2}_{x}}
$
while we use the optimal bound by $\|F\big\|_{L^{2}_{t}L^{\frac{6}{5}}_{x}}$ (for $n=3$).

The paper is organized as follows: Section 2 is devoted to the endpoint Strichartz estimates for scalar charge transfer model. In section 3, we will first prove the asymptotic completeness for matrix charge transfer model and apply it to show the endpoint Strichartz estimates of matrix case.

{\bf Acknowledgements:} The first author is supported by NSFC (No. 11301203) and the Fundamental Research Funds for the Central Universities (CCNU-14A05037). The second author is partially supported by NSF DMS-1201394 and NSFC (No. 11671163).  The third author is supported by NSFC (No. 11371158) and the program for Changjiang Scholars and Innovative Research Team in University (No. IRT13066). Part of this work was done when the first author visited the Rutgers University and the second author visited  CCNU.

\section{The Strichartz estimates for scalar charge transfer model}

\setcounter{equation}{0}
\subsection{The projections $ P_c(s)$ and asymptotic completeness}
%\begin{rem}\label{rem-1}{\rm As explained in \cite{Cai}, the definition makes sense for initial data  $\psi_{0}\in L^{p}$ with $p\in [1,2]$. Furthermore, we have the following.
%
%(i) Actually in Proposition 3.1 of \cite{RSS1}, the authors proved
%\begin{eqnarray}\label{eq-6}
%\|P_{b}(H_{1})U(t)\psi_{0}\|_{L^2}+\|P_{b}(H_{2},t)U(t)\psi_{0}\|_{L^2}\leq e^{-\alpha t}\|\psi_{0}\|_{L^2}
%\end{eqnarray}
%for some $\alpha>0$ if $\psi_{0}$ satisfies (\ref{eq-4}).
%
%(ii) For every $t$, $P_{b}(H_{2},t)=G_{-\vec{e_{1}}}(t)P_{b}(H_{2})G_{\vec{e_{1}}}(t)$ is an orthogonal projection. It gives the projection onto functions which are the bound states of $H_{2}$ being translated to the position of the potential $V_{2}(\cdot-t\vec{e_{1}})$. We can give an explicit formula of $P_{b}(H_{2},t)$ by
%\begin{eqnarray}
%P_{b}(H_{2},t)f(x)=\sum^{\ell}_{j=1}e^{ix_{1}} w_{j}(x-t\vec{e_{1}})\int_{\mathbb{R}^{n}}f(y)\overline{e^{iy_{1}} w_{j}(y-t\vec{e_{1}})}dy.\nonumber
%\end{eqnarray}
%
%%iii) The set of $\psi_{0}$ satisfying (\ref{eq-4}) is a closed subset of $L^{2}$, which coincides with the space of \emph{scattering states} for the charge transfer problem. The latter was well defined in Theorem 1.1 of Graf \cite{Gr}. We will discuss it in more detail in the next subsection.
%}
%\end{rem}

We begin with the results on asymptotic completeness of the charge transfer problem (\ref{main-equation}). Let $s\in \mathbb{R}$ and
\begin{eqnarray}
\Omega^{-}_{0}(s)=s-\lim_{t\rightarrow+\infty}U(s,t)e^{-iH_{0}(t-s)},\nonumber
\end{eqnarray}
\begin{eqnarray}
\Omega^{-}_{1}(s)=s-\lim_{t\rightarrow+\infty}U(s,t)e^{-iH_{1}(t-s)}P_{b}(H_{1}),\nonumber
\end{eqnarray}
\begin{eqnarray}
\Omega^{-}_{2}(s)=s-\lim_{t\rightarrow+\infty}U(s,t)G_{-\vec{e_{1}}}(t)e^{-iH_{2}(t-s)}P_{b}(H_{2})G_{\vec{e_{1}}}(s)\nonumber
\end{eqnarray}
be the wave operators, where $H_{0}=-\frac{1}{2}\Delta$ and $U(t,s)$ is the propagator of (\ref{main-equation}). Then one has the following conclusion (See e.g. \cite{Gr}).
\begin{pro}\label{ex-pro-2}
Let $s\in \mathbb{R}$, the following assertions hold.

{\rm (i)} For each $s\in \mathbb{R}$, the above wave operators $\Omega^{-}_{\kappa}(s)$ $(\kappa=0,1,2)$ exist in $L^2$.

{\rm (ii)}  The ranges
Ran$\Omega^{-}_{\kappa}(s)$ ($\kappa=0,1,2$) are closed and orthogonal to each other.

{\rm (iii)}  The asymptotic completeness:
\begin{eqnarray}
L^{2}={\rm Ran}\Omega^{-}_{0}(s)\oplus{\rm Ran}\Omega^{-}_{1}(s)\oplus{\rm Ran}\Omega^{-}_{2}(s).\nonumber
\end{eqnarray}
\end{pro}
One could see  \cite{Gr, Ya1, Wu, Zi} for the detail of the scattering existences and asymptotic completeness of such model. Recall that the set $\mathscr{A}(s)$ consists of all $f$ such that
\begin{eqnarray}\label{eq-7}
\|P_{b}(H_{1})U(t,s)f\|_{L^2}+\|P_{b}(H_{2},t)U(t,s)f\|_{L^2}\rightarrow0\ \ {\rm as} \ \ t\rightarrow +\infty.
\end{eqnarray}
The following results show that $\mathscr{A}(s)$ is equal to the closed range of first wave operator $\Omega^{-}_{0}(s)$, and the projections $P_c(s)$ on $\mathscr{A}(s)$ have some favorable properties.

\begin{pro}\label{ex-le-1}
Let $t, s\in \mathbb{R}$. Then

{\rm (i)} $\mathscr{A}(s)={\rm Ran}\Omega^{-}_{0}(s)$.

{\rm (ii)} The propagator $U(s,t)$ propagates $\mathscr{A}(t)$
to $\mathscr{A}(s)$, that is,
\begin{eqnarray}\label{commute}
U(s,t)\mathscr{A}(t)=\mathscr{A}(s).
\end{eqnarray}

{\rm (iii)} The propagator $U(s,t)$ commutes with $P_{c}(t)$, that is,
\begin{eqnarray}\label{commute-1}
U(s,t)P_{c}(t)=P_{c}(s)U(s,t).
\end{eqnarray}
\end{pro}

\begin{proof}
We first prove (i). It follows from the proof of \cite[Theorem 4.1]{RSS1} that
$$L^{2}=\mathscr{A}(0)+{\rm Ran}\Omega^{-}_{1}(0)+{\rm Ran}\Omega^{-}_{2}(0).$$
Actually, the above identity holds for all $s\in \mathbb{R}$ by the same procedure.
Moreover, let $\psi_{0}\in \mathscr{A}(s)$, then for any $g\in L^{2}$
\begin{eqnarray}
\langle \psi_{0}, \Omega^{-}_{1}(s)g \rangle
=\lim_{t\rightarrow\infty}\langle P_{b}(H_{1})U(s,t)\psi_{0}, e^{-i(t-s)H_{1}}g \rangle=0,\nonumber
\end{eqnarray}
which means $\psi_{0}\perp {\rm Ran}\Omega^{-}_{1}(s)$.
Similarly, $\psi_{0}\perp {\rm Ran}\Omega^{-}_{2}(s)$, which means the orthogonal sum decomposition
$$L^{2}=\mathscr{A}(s)\oplus{\rm Ran}\Omega^{-}_{1}(s)\oplus{\rm Ran}\Omega^{-}_{2}(s).$$
Thus, by (iii) of Proposition \ref{ex-pro-2}, we have  $\mathscr{A}(s)={\rm Ran}\Omega^{-}_{0}(s)$.

Now turn to (ii). Notice that $P_{b}(H_{1})U(r,t)f=P_{b}(H_{1})U(r,s)U(s,t)f$, thus by definition of $\mathscr{A}(t)$ and $\mathscr{A}(s)$, it is easy to check  $U(s,t)\mathscr{A}(t)=\mathscr{A}(s)$.

Finally, we prove (iii).
Notice that $U(s,t)\Omega^{-}_{\kappa}(t)=\Omega^{-}_{\kappa}(s)e^{-iH_{\kappa}(s-t)}$, which means that $U(s,t)$ propagates ${\rm Ran}\Omega^{-}_{\kappa}(t)$ to ${\rm Ran}\Omega^{-}_{\kappa}(s)$ ($\kappa=1,2$). Then by (ii) of proposition \ref{ex-pro-2} that
\begin{eqnarray}\label{ex-eq-3}
P_{c}(s)U(s,t)P_{\kappa b}(t)=0,\ \kappa=1,2.\nonumber
\end{eqnarray}
Then by (\ref{commute}) and (ii) of Proposition \ref{ex-pro-2},
$$P_{c}(s)U(s,t)=P_{c}(s)U(s,t))(P_{c}(t)+P_{1b}(t)+P_{2b}(t))=P_{c}(s)U(s,t)P_{c}(t)=U(s,t)P_{c}(t).$$

\end{proof}

Notice that $\Omega^{-}_{1}(s)=\Omega^{-}_{1}(s)P_{b}(H_{1})$ by definition and $\Omega^{-}_{1}(s)$ is isometry,
then the set
$$\big\{\tilde{u}_{j}(s)=\Omega^{-}_{1}(s)u_{j}\big\}^{m}_{j=1}$$
is the basis of linear space $Ran\Omega^{-}_{1}(s)$, where $u_{j}$ are bound states of $H_{1}$. Moreover,
 we have the following estimate for this group of bases $\tilde{u}_{j}$.
\begin{lem}\label{thm0}
Let $\tilde{u}_{j}(s)$ $(1\leq j\leq m)$ be defined as above. Then for $\sigma\geq0$ and multi-index $\gamma$ with $|\gamma|\geq0$,
$$\big\|\langle x\rangle^{\sigma}\partial^{\gamma}_{x}\tilde{u}_{j}(s)\big\|_{L^{2}},\ \ \ \ 1\leq j\leq m$$
 is uniformly bounded in $s$.
\end{lem}

\begin{proof}
 First of all, for each $1\leq j\leq m$, it is easy to see that $\|\tilde{u}_{j}(s)\|_{L^{2}}$ is uniformly bounded in $s$.
Notice that it follows from the Duhamel formula that
\begin{eqnarray}\label{eq-21}
U(s,t)e^{-iH_{1}(t-s)}u_{j}&=&u_{j}+i\int^{t}_{s}U(s,r)V_{2}(\cdot-\vec{e}_{1}r)e^{-iH_{1}(r-s)}u_{j}dr\nonumber\\
&=&u_{j}+i\int^{t}_{s}U(s,r)V_{2}(\cdot-\vec{e}_{1}r)e^{-i\lambda_{j}(r-s)}u_{j}dr
\end{eqnarray}
Since $V_{1}$ and $V_{2}$ are exponentially localized smooth functions, the function $u_{j}$ is also smooth and exponentially localized in $L^{2}$, thus $\langle x\rangle^{\sigma}\partial^{\gamma}_{x}u_{j}\in L^{2}$ for all $\sigma>0$ and multi-index $\gamma$.  On the other hand, the function
$$K_{j}(r,s,x):=V_{2}(x-\vec{e}_{1}r)e^{-i\lambda_{j}(r-s)}u_{j}(x)$$
has the property that for any $\sigma>0$, multi-index $\beta$ and $N>0$
$$\big\|\langle x\rangle^{\sigma}\partial^{\beta}_{x}K_{j}(r,s,\cdot)\big\|_{L_{x}^{2}}\leq c(\sigma,|\beta|,j,N)\langle r\rangle^{-N}.$$
Now notice that for any integer $\sigma>0$ and multi-index $\gamma$, it has been proved in \cite[p. 148]{RSS1} that
\begin{eqnarray}\label{eq-22}
\big\|\langle x\rangle^{\sigma} \partial^{\gamma}_{x}U^{\pm 1}(t)g\big\|_{L^{2}}\lesssim \langle t\rangle^{3\sigma+|\gamma|}\sum_{|\beta|\leq \sigma+|\gamma|}\big\|\langle x\rangle^{\sigma} \partial^{\beta}_{x}g\big\|_{L^{2}},\nonumber
\end{eqnarray}
which implies
\begin{eqnarray}\label{eq-23}
\big\|\langle x\rangle^{\sigma}\partial^{\gamma}_{x}U(s,r)K_{j}(r,s,\cdot)\big\|_{L_{x}^{2}}
&=&\big\|\langle x\rangle^{\sigma}\partial^{\gamma}_{x}U(s)U^{-1}(r)K_{j}(r,s,\cdot)\big\|_{L_{x}^{2}}\nonumber\\
&\lesssim& \langle s\rangle^{3\sigma+|\gamma|}\sum_{|\beta|\leq \sigma+|\gamma|}\big\|\langle x\rangle^{\sigma}\partial_{x}^{\beta}U^{-1}(r)K_{j}(r,s,\cdot)\big\|_{L_{x}^{2}}\nonumber\\
&\lesssim& \langle s\rangle^{3\sigma+|\gamma|}\langle r\rangle^{4\sigma+|\gamma|}\sum_{|\beta'|\leq 2\sigma+|\gamma|}\big\|\langle x\rangle^{\sigma}\partial_{x}^{\beta'}K_{j}(r,s,\cdot)\big\|_{L_{x}^{2}}\nonumber\\
&\lesssim&  \langle s\rangle^{3\sigma+|\gamma|}\langle r\rangle^{4\sigma+|\gamma|-N}.
\end{eqnarray}
Then by (\ref{eq-21})-(\ref{eq-23}) and choosing $N>7\sigma+2|\gamma|+2$, we could obtain that for all $\gamma$ and positive integer $\sigma$,
$$\big\|\langle x\rangle^{\sigma}\partial^{\gamma}_{x}\tilde{u}_{j}(s)\big\|_{L^{2}}$$
is uniformly bounded in $s$. For arbitrary $\sigma>0$,  it can be proved by using interpolation. Hence we finish the proof.
\end{proof}
\vskip0.3cm
\begin{rem}\label{rem-2}{\rm (i) Since
$$U(s,t)G_{-\vec{e_{1}}}(t)e^{-iH_{2}(t-s)}P_{b}(H_{2})G_{\vec{e_{1}}}(s)=G_{-\vec{e_{1}}}(s)\widetilde{U}(s,t)e^{-iH_{2}(t-s)}P_{b}(H_{2})G_{\vec{e_{1}}}(s),$$
where $\widetilde{U}(t)$ is the propagator  of $i\partial_{t}\psi=\widetilde{H}\psi$ with $\widetilde{H}=-\frac{1}{2}\Delta+V_{1}(x+\vec{e}_{1}t)+V_{2}$. On the other hand, one always has the identity $$\Omega^{-}_{2}(s)=\Omega^{-}_{2}(s)P_{b}(H_{2},s).$$
By the same argument, we could show that Lemma \ref{thm0} still holds for
$\widetilde{\Omega}^{-}_{2}(s)w_{j}$
with $$\widetilde{\Omega}^{-}_{2}(s)=s-\lim_{t\rightarrow+\infty}\widetilde{U}(s,t)e^{-iH_{2}(t-s)}P_{b}(H_{2}),$$
where $w_{j}$ $(1\leq j\leq \ell)$ is the bound states of $H_{2}$ defined as before. Moreover, it is easy to see that
\begin{eqnarray}\label{eq-27}
\big\{\tilde{w}_{j}(s)=G_{-\vec{e_{1}}}(s)\widetilde{\Omega}^{-}_{2}(s)w_{j}G_{\vec{e_{1}}}(s)\big\}^{\ell}_{j=1},
\end{eqnarray}
is the basis of $Ran\Omega^{-}_{2}(s)$.

(ii) In fact, we can see from (\ref{eq-23}) that $\tilde{u}_{j}(s)$ is close to $u_{j}$ in $L^{2}$ as $s$ gets large. Moreover, for arbitrary $N>0$,
$$\|\tilde{u}_{j}(s)-u_{j}\|_{L^{2}}\lesssim \langle s\rangle^{-N}.$$

(iii) By choosing suitable $\sigma$ and $\gamma$ in Lemma \ref{thm0} and the Sobolev embedding, we could obtain that $\tilde{u}_{j}(s)\in L^{p}$ for all $1\leq p\le\infty$. Then it follows that the wave operator  $\Omega^{-}_{1}(s)$ is uniformly  bounded on $L^{p}$ for all $1\le p\le \infty$. Similarly, the same conclusion holds for $\Omega^{-}_{2}(s)$.
}
\end{rem}

Since $\mathscr{A}(s)={\rm Ran}\Omega^{-}_{0}(s)$, the projection $P_{c}(s)$ onto the space $\mathscr{A}(s)$ is actually a projection on the scattering states of the charge transfer transfer model.  For the spaces ${\rm Ran}\Omega^{-}_{1}(s)$ and ${\rm Ran}\Omega^{-}_{2}(s)$, we naturally considered them as its ``bounded states" and  use $P_{1b}(s)$ and $P_{2b}(s)$ to denote the projections onto
${\rm Ran}\Omega^{-}_{1}(s)$ and ${\rm Ran}\Omega^{-}_{2}(s)$, respectively.  Thus the asymptotic completeness says that for all $s\in \mathbb{R}$
\begin{eqnarray}\label{eq-26}
P_{c}(s)+P_{1b}(s)+P_{2b}(s)=I
\end{eqnarray}
on $L^2$.  Now we will prove the following uniformly weighted estimates for projections $P_{jb}(s)\ (j=1,2)$, which are very important to control $P_c(s)$ and establish endpoint Strichartz estimates of the charge transfer model.

\begin{pro}\label{lem-1}Let $\sigma_{1}$ and $\sigma_{2}$ be any nonnegative numbers. Then we have
\begin{eqnarray}
\sup_{s}\big\|\langle x-D(s)\rangle^{-\sigma_{1}}P_{1b}(s)\langle x\rangle^{\sigma_{2}}f\big\|_{L_{x}^{2}}\le C\|f||_{L_{x}^{2}},\nonumber
\end{eqnarray}
and
\begin{eqnarray}
\sup_{s}\big\|\langle x-D(s)\rangle^{-\sigma_{1}}P_{2b}(s)\langle x-\vec{e}_{1}s\rangle^{\sigma_{2}}f\big\|_{L_{x}^{2}}\le C\|f||_{L_{x}^{2}}.\nonumber
\end{eqnarray}
where  $D(t)$ denotes either $0$ or $\vec{e}_{1}t$.
\end{pro}

\begin{proof}
Since $P_{1b}(s)$ is the projection onto $Ran\Omega^{-}_{1}(s)$ and $\big\{\tilde{u}_{j}(s)\big\}^{m}_{j=1}$ is the basis of $Ran\Omega^{-}_{1}(s)$, let us just assume it is orthogonal basis, otherwise one could use Schmidt's orthogonalization which will not affect the estimates here. we have
\begin{eqnarray}\label{eq-24}
P_{1b}(s)f=\sum^{m}_{i=1}\langle f,\tilde{u}_{j}(s)\rangle \tilde{u}_{j}(s),
\end{eqnarray}
where $\tilde{u}_{j}(s)$ $(1\leq j\leq m)$ is defined as in Lemma \ref{thm0}.
Then for $f\in L^{2}$ and any $\sigma_{1}, \sigma_{2}>0$, it follows from Lemma \ref{thm0} that
\begin{eqnarray}
&&\big\|\langle x-D(s)\rangle^{-\sigma_{1}}P_{1b}(s)\langle x\rangle^{\sigma_{2}}f\big\|_{L^{2}}\nonumber\\
&\leq& \|\langle x-D(s)\rangle^{-\sigma_{1}}\|_{L^{\infty}}
\sum^{m}_{j=1}\big|\langle \langle x\rangle^{\sigma_{2}}f,\tilde{u}_{j}(s)\rangle\big| \|\tilde{u}_{j}(s)\|_{L^{2}}\nonumber\\
&\leq &\|\langle x\rangle^{-\sigma_{1}}\|_{L^{\infty}}\|f\|_{L^{2}}\sum^{m}_{j=1}\|\langle x\rangle^{\sigma_{2}}\tilde{u}_{j}(s)\|_{L^{2}}\|\tilde{u}_{j}(s)\|_{L^{2}}\nonumber\\
&\leq& C\|f\|_{L^{2}},\nonumber
\end{eqnarray}
where the constant $C>0$ is independent of $s$.

Notice that  the projection $P_{2b}(s)$ is of form
\begin{eqnarray}\label{eq-25}
P_{2b}(s)f=\sum^{m}_{i=1}\langle f,\tilde{w}_{j}(s)\rangle \tilde{w}_{j}(s),
\end{eqnarray}
where $\tilde{w}_{j}(s)$ ($1\leq j\leq \ell$) is defined by (\ref{eq-27}) in (i) of Remark \ref{rem-2}. Thus for
 $f\in L^{2}$ and any $\sigma_{1}, \sigma_{2}>0$,
\begin{eqnarray}
&&\big\|\langle x-D(s)\rangle^{-\sigma_{1}}P_{2b}(s)\langle x-\vec{e}_{1}s\rangle^{\sigma_{2}}f\big\|_{L^{2}}\nonumber\\
&\leq& \|\langle x-D(s)\rangle^{-\sigma_{1}}\|_{L^{\infty}}
\sum^{\ell}_{j=1}\big|\langle \langle x-\vec{e}_{1}s\rangle^{\sigma_{2}}f,\tilde{w}_{j}(s)\rangle\big| \|\tilde{w}_{j}(s)\|_{L^{2}}\nonumber\\
&\leq &\|\langle x\rangle^{-\sigma_{1}}\|_{L^{\infty}}\|f\|_{L^{2}}\sum^{\ell}_{j=1}\big\|G_{-\vec{e_{1}}}(s)\langle x\rangle^{\sigma_{2}}\widetilde{\Omega}^{-}_{2}(s)w_{j}G_{\vec{e_{1}}}(s)\big\|_{L^{2}}\|\tilde{w}_{j}(s)\|_{L^{2}}\nonumber\\
&\leq& C\|f\|_{L^{2}},\nonumber
\end{eqnarray}
where (i) of Remark \ref{rem-2} is applied and the constant $C>0$ is independent of $s$.
\end{proof}

\begin{rem}{\rm
By (\ref{eq-26}), (\ref{eq-24}) and (iii) of Remark \ref{rem-2}, it is easy to prove that $P_{c}(s)$ and $P_{\kappa b}(s)$ ($\kappa=1,2$) are uniformly bounded on $L^p$ for all $1\le p\le \infty$.
}
\end{rem}

\subsection{The proofs of Strichartz estimates}
%Theorems \ref{thm1} and \ref{thm2}
%We will use a different way than Keel-Tao to derive the homogeneous endpoint Strichartz estimates. It is worth to point out that this method could also be used to deal with the matrix charge transfer problem.
%Recall that The admissible pair $(p,q)$ for Strichartz estimates satisfies
%$$
%\frac{2}{p}=\frac{n}{2}-\frac{n}{q},\ \ \ \ \ 2\leq p\leq\infty.
%$$
In this subsection, we will prove Theorems \ref{thm1} and \ref{thm2}, i.e. full Strichartz estimates for the scalar charge transfer problem (\ref{main-equation}).  Let us first recall the optimal Strichartz estimates for the free Schr\"odinger equation. It is well-known that (see e.g Keel-Tao \cite{KT})
\begin{eqnarray}\label{s-f}
\big\|e^{-it\Delta/2}f\big\|_{L_{t}^{p}L_{x}^{q}}\lesssim\|f\|_{L^{2}},
\end{eqnarray}
\begin{eqnarray}\label{s-f'}
\Big\|\int_{0}^t e^{-i(t-s)\Delta/2}f(x,s)ds\Big\|_{L_{t}^{p}L_{x}^{q}}\lesssim\|f\|_{L_{t}^{\tilde{p}'}L_{x}^{\tilde{q}'}},
\end{eqnarray}
where $(p,q)$ and $(\tilde{p},\tilde{q})$ are admissible pairs satisfying the (\ref{pair}) and $\frac{1}{\tilde{p}'}+\frac{1}{\tilde{p}}=1$.

\vskip0.3cm

{\bf The proof of Theorem \ref{thm1}:}
we only show the result for $n=3$ without the loss of generality. The proof can be concluded by the following three steps.

\emph{Step 1, Kato-Jensen estimates.} We will show that for $0\leq t_{0}<t$ and $\sigma>\frac{3}{2}$,
\begin{eqnarray}\label{eq-9}
\big\|\langle x-D(t)\rangle^{-\sigma}U(t,t_{0})P_{c}(t_{0})\langle x-D(t_{0})\rangle^{-\sigma}\big\|_{L^{2}\rightarrow L^{2}}\leq C\langle t-t_{0}\rangle^{-\frac{3}{2}}
\end{eqnarray}
where the constant $C$ is independent of $t$ and $t_{0}$.
To this end,  recall that from Proposition \ref{ex-pro-1} the decay estimates
\begin{eqnarray}\label{eq-8}
\big\|U(t,t_{0})f\big\|_{L^{\infty}}\lesssim |t-t_{0}|^{-\frac{3}{2}} \|f\|_{L^{1}}\nonumber
\end{eqnarray}
hold for $0\leq t_{0}<t$ and $f\in L^{1}\cap \mathscr{A}(t_{0})$, where we use the assumption that  $V_{1}$ and $V_{2}$ are exponentially localized smooth functions (see Definition \ref{defi-1}).
For $|t-t_{0}|\leq 1$, (\ref{eq-9}) follows by $\|U(t,t_{0})\|_{L^{2}\rightarrow L^{2}}\leq 1$.
For $|t-t_{0}|>1$, we have
\begin{eqnarray}
&&\big\|\langle x-D(t)\rangle^{-\sigma}U(t,t_{0})P_{c}(t_{0})\langle x-D(t_{0})\rangle^{-\sigma}f\big\|_{L^{2}\rightarrow L^{2}}\nonumber\\
&\leq& \big\|\langle x-D(t)\rangle^{-\sigma}\big\|_{L^{2}}\|U(t,t_{0})P_{c}(t_{0})\|_{{L^{1}}\rightarrow L^{\infty}}\big\|\langle x-D(t_{0})\rangle^{-\sigma}f\big\|_{L^{1}}\nonumber\\
&\leq& C|t-t_{0}|^{-\frac{3}{2}}\|f\|_{L^2}\nonumber
\end{eqnarray}
where the constant $C$ is independent $t$ and $t_{0}$. On the other hand, by (\ref{commute-1}), we also have
\begin{eqnarray}\label{eq-9''}
\big\|\langle x-D(t)\rangle^{-\sigma}P_{c}(t)U(t,t_{0})\langle x-D(t_{0})\rangle^{-\sigma}\big\|_{L^{2}\rightarrow L^{2}}\leq C\langle t-t_{0}\rangle^{-\frac{3}{2}}
\end{eqnarray}

\emph{Step 2, local decay estimates.} We intend to prove that for $\psi_{0}\in \mathscr{A}(0)$ and $\sigma>\frac{3}{2}$
\begin{eqnarray}\label{eq-10}
\big\|\langle x-D(t)\rangle^{-\sigma}U(t)\psi_{0}\big\|_{L_{t}^{2}L_{x}^{2}}\leq C\|\psi_{0}\|_{L^2}.
\end{eqnarray}
In fact, since $U(t)\psi_{0}$ belongs to $\mathscr{A}(t)$, it is equivalent to prove
$$\big\|\langle x-D(t)\rangle^{-\sigma}P_{c}(t)U(t)\psi_{0}\big\|_{L_{t}^{2}L_{x}^{2}}\leq C\|\psi_{0}\|_{L^2}.$$
Consider the Cauchy problem
\begin{eqnarray}\label{eq-11-1}
&&i\Phi_{t}=-\frac{1}{2}\Delta\Phi=(-\frac{1}{2}\Delta+V_{1}+V_{2}(x-\vec{e}_{1}t))\Phi-(V_{1}+V_{2}(x-\vec{e}_{1}t))\Phi\nonumber\\
&&\Phi(0,\cdot)=\psi_{0}.\nonumber
\end{eqnarray}
By Duhamel's formula and $\Phi(t)=e^{it\frac{\Delta}{2}}\psi_{0}$, it follows that
\begin{eqnarray}\label{eq-11-2}
P_{c}(t)U(t)\psi_{0}=P_{c}(t)e^{it\frac{\Delta}{2}}\psi_{0}-i\int^{t}_{0}P_{c}(t)U(t,s)(V_{1}+V_{2}(\cdot-\vec{e}_{1}s))e^{is\frac{\Delta}{2}}\psi_{0}ds.
\end{eqnarray}
Notice that by H\"older inequality and  (\ref{s-f}) (i.e. the endpoint Strichartz estimates for $e^{it\Delta/2}$),
\begin{eqnarray}\label{eq-12-1}
\big\|\big(|V_{1}|^{\frac{1}{2}}+|V_{2}(\cdot-\vec{e}_{1}t)|^{\frac{1}{2}}\big)e^{it\frac{\Delta}{2}}\psi_{0}\big\|_{L_{t}^{2}L_{x}^{2}}
&\leq&\Big(\big\||V_{1}|^{\frac{1}{2}}\big\|_{L^{3}}+\big\||V_{2}|^{\frac{1}{2}}\big\|_{L^{3}}\Big)\ \big\|e^{it\frac{\Delta}{2}}\psi_{0}\big\|_{L_{t}^{2}L_{x}^{6}}\nonumber\\
&\leq& C\|\psi_{0}\|_{L^{2}},
\end{eqnarray}
By combining  the Kato-Jensen estimates (\ref{eq-9''}) and the Young inequality, (\ref{eq-12-1}) implies that
\begin{eqnarray}\label{eq-12}
&&\Big\|\int^{t}_{0}\langle x-D(t)\rangle^{-\sigma}P_{c}(t)U(t,s)(V_{1}+V_{2}(\cdot-\vec{e}_{1}s))e^{is\frac{\Delta}{2}}\psi_{0}ds\Big\|_{L_{t}^{2}L_{x}^{2}}\nonumber\\
&\leq& \Big\|\int^{t}_{0}\big\|\langle x-D(t)\rangle^{-\sigma}P_{c}(t)U(t,s)\langle x\rangle^{-\sigma}\langle x\rangle^{\sigma}|V_{1}|^{\frac{1}{2}}|V_{1}|^{\frac{1}{2}}e^{is\frac{\Delta}{2}}\psi_{0}\big\|_{L_{x}^{2}}ds\Big\|_{L_{t}^{2}}\nonumber\\
&&\ \ \ \ \ \  +\ \Big\|\int^{t}_{0}\big\|\langle x-D(t)\rangle^{-\sigma}P_{c}(t)U(t,s)\langle x-\vec{e}_{1}s\rangle^{-\sigma}\nonumber\\
&&\ \ \ \ \ \ \ \ \ \ \ \ \ \ \ \ \ \ \ \ \ \ \times\langle x-\vec{e}_{1}s\rangle^{\sigma}|V_{2}(\cdot-\vec{e}_{1}s)|^{\frac{1}{2}}|V_{2}(\cdot-\vec{e}_{1}s)|^{\frac{1}{2}}e^{is\frac{\Delta}{2}}\psi_{0}\big\|_{L_{x}^{2}}ds\Big\|_{L_{t}^{2}}\nonumber\\
&\leq&C\Big\|\int^{t}_{0}\langle t-s\rangle^{-\frac{3}{2}}\Big(\big\||V_{1}|^{\frac{1}{2}}e^{is\frac{\Delta}{2}}\psi_{0}\big\|_{L_{x}^{2}}
+\big\||V_{2}(\cdot-\vec{e}_{1}s)|^{\frac{1}{2}}e^{is\frac{\Delta}{2}}\psi_{0}\big\|_{L_{x}^{2}}\Big)ds\Big\|_{L_{t}^{2}}\nonumber\\
&\leq& C\|\psi_{0}\|_{L^2}.
\end{eqnarray}
On the other hand, note that (\ref{eq-12-1}) still holds if the potentials $|V_{1}|^{\frac{1}{2}}$ and $|V_{2}(\cdot-\vec{e}_{1}s)|^{\frac{1}{2}}$ replace by  $\langle x-D(t)\rangle^{-\sigma}$ for $\sigma>\frac{3}{2}$, and the identity $$P_{c}(t)+P_{1b}(t)+P_{2b}(t)=I,$$
then it follows from Proposition  \ref{lem-1} that
\begin{eqnarray}\label{eq-13}
&&\big\|\langle x-D(t)\rangle^{-\sigma}P_{c}(t)e^{it\frac{\Delta}{2}}\psi_{0}\big\|_{L_{t}^{2}L_{x}^{2}}\nonumber\\
&\leq& \big\|\langle x-D(t)\rangle^{-\sigma}e^{it\frac{\Delta}{2}}\psi_{0}\big\|_{L_{t}^{2}L_{x}^{2}}+\sum^{2}_{\kappa=1}\big\|\langle x-D(t)\rangle^{-\sigma}P_{\kappa b}(t)e^{it\frac{\Delta}{2}}\psi_{0}\big\|_{L_{t}^{2}L_{x}^{2}}\nonumber\\
&\leq&C\|\psi_{0}\|_{L^{2}}+\sup_{t}\big\|\langle x-D(t)\rangle^{-\sigma}P_{1b}(t)\langle x\rangle^{\sigma}\big\|_{L_{x}^{2}\rightarrow L_{x}^{2}}\big\|\langle x\rangle^{-\sigma}e^{it\frac{\Delta}{2}}\psi_{0}\big\|_{L_{t}^{2}L_{x}^{2}}\nonumber\\
&&\ \ \ \ \ \ \ +\sup_{t}\big\|\langle x-D(t)\rangle^{-\sigma}P_{2b}(t)\langle x-\vec{e}_{1}t\rangle^{\sigma}\big\|_{L_{x}^{2}\rightarrow L_{x}^{2}}\big\|\langle x-\vec{e}_{1}t\rangle^{-\sigma}e^{it\frac{\Delta}{2}}\psi_{0}\big\|_{L_{t}^{2}L_{x}^{2}}\nonumber\\
&\leq&C\|\psi_{0}\|_{L^{2}}.
\end{eqnarray}
Therefore, by (\ref{eq-11-2})-(\ref{eq-13}), we obtain
\begin{eqnarray}\label{eq-28}
\|\langle x-D(t)\rangle^{-\sigma}P_{c}(t)U(t)\psi_{0}\|_{L_{t}^{2}L_{x}^{2}}\leq C\|\psi_{0}\|_{L^2}.\nonumber
\end{eqnarray}

\emph{Step 3, The homogeneous Strichartz estimates.} Now let $\psi(t)=U(t)\psi_{0}$ be the solution of the equation (\ref{main-equation}). By Duhamel's formula, we have
\begin{eqnarray}
\psi(t)=U(t)\psi_{0}=e^{it\frac{\Delta}{2}}\psi_{0}-i\int^{t}_{0}e^{i(t-s)\frac{\Delta}{2}}(V_{1}+V_{2}(\cdot-\vec{e}_{1}s))U(s)\psi_{0}ds.
\end{eqnarray}
Let $(p,q)$ be admissible pair, it follows from (\ref{s-f}) and (\ref{eq-10}) that
\begin{eqnarray}
\|U(t)\psi_{0}\|_{L_{t}^{p}L_{x}^{q}}&\leq& \big\|e^{it\frac{\Delta}{2}}\psi_{0}\big\|_{L_{t}^{p}L_{x}^{q}}+\Big\|\int^{t}_{0}e^{i(t-s)\frac{\Delta}{2}}(V_{1}+V_{2}(\cdot-\vec{e}_{1}s))U(s)\psi_{0}ds\Big\|_{L_{t}^{p}L_{x}^{q}}\nonumber\\
&\leq& C\Big(\|\psi_{0}\|_{L^{2}}+\big\|(V_{1}+V_{2}(\cdot-\vec{e}_{1}t))U(t)\psi_{0}\big\|_{L_{t}^{2}L_{x}^{\frac{6}{5}}}\Big)\nonumber\\
&\leq& C\Big(\|\psi_{0}\|_{L^{2}}+\big\||V_{1}|^{\frac{1}{2}}\langle x\rangle^{\sigma}\langle x\rangle^{-\sigma}U(t)\psi_{0}\big\|_{L_{t}^{2}L_{x}^{2}}\nonumber\\
&&\ \ \ \ \ \ \ \ \ \ \ \ +\
\big\||V_{2}(\cdot-\vec{e}_{1}t)|^{\frac{1}{2}}\langle x-\vec{e}_{1}t\rangle^{\sigma}\langle x-\vec{e}_{1}t\rangle^{-\sigma}U(t)\psi_{0}\big\|_{L_{t}^{2}L_{x}^{2}}\Big)\nonumber\\
&\leq& C\|\psi_{0}\|_{L^{2}}.\nonumber
\end{eqnarray}
Hence we finish the whole proof.

\vskip0.4cm

Now consider the nonhomogeneous charge transfer model (\ref{main-equation-21} ), i.e.
$$
i\partial_{t}\psi=-\frac{1}{2}\Delta\psi+V_{1}\psi+V_{2}(x-\vec{e}_{1}t)\psi+F(t),\ \ \psi(0,\cdot)=\psi_{0}\in \mathscr{A}(0).
$$
We will prove Theorem \ref{thm2}.  For this, we first project the both sides of the equation onto the scattering states:
\begin{eqnarray}\label{eq-14}
iP_{c}(t)\partial_{t}\psi=P_{c}(t)H(t)\psi+P_{c}(t)F(t)
\end{eqnarray}
Notice that by differentiation on the both side of the (\ref{commute-1}) with respect to $s$ at $s=t$, we obtain
\begin{eqnarray}
i\dot{P_{c}}(t)\psi=H(t)P_{c}(t)\psi-P_{c}(t)H(t)\psi.\nonumber
\end{eqnarray}
Hence the
(\ref{eq-14})
is equivalent to
\begin{eqnarray}\label{eq-14'}
i\partial_{t}(P_{c}(t)\psi)=H(t)P_{c}(t)\psi+P_{c}(t)F(t).
\end{eqnarray}

%\begin{rem}{\rm If $F=0$ and the initial data $\psi_{0}\in \mathscr{A}(0)$, the above identity means that the projection is naturally existed in equation (\ref{main-equation}) since the solution $\psi(t)=P_{c}(t)\psi(t)$.
%}
%\end{rem}

%\begin{thm}\label{thm2}
%Consider the equation (\ref{main-equation-21}) with initial date $\psi_{0}\in \mathscr{A}(0)$. Let $(p,q)$ and $(\tilde{p},\tilde{q})$ be admissible pair, then the solution $\psi(t,x)$ satisfies the following  Strichartz estimates.
%\begin{eqnarray}\label{NHS}
%\|P_{c}(t)\psi(t)\|_{L^{p}_{t}L^{q}_{x}}\lesssim \|\psi_{0}\|_{L^{2}}+\|F\|_{L^{\tilde{p}'}_{t}L^{\tilde{q}'}_{x}}.
%\end{eqnarray}
%\end{thm}
{\bf The proof of Theorem \ref{thm2}:} Similarly, we only prove the estimates for $n=3$ and  divide the proof into serval steps. The first two steps have been proved in
Theorem \ref{thm1}. Denote by $U(t,s)$ the propagator of the linear equation $i\partial_{t}\psi=H(t)\psi$ and also write $U(t)=U(t,0)$.

\emph{Step 1, Kato-Jensen estimates.} For $0\leq t_{0}<t$ and $\sigma>\frac{3}{2}$,
\begin{eqnarray}\label{eq-9-2}
\big\|\langle x-D(t)\rangle^{-\sigma}U(t,t_{0})P_{c}(t_{0})\langle x-D(t_{0})\rangle^{-\sigma}\big\|_{L^{2}\rightarrow L^{2}}\leq C\langle t-t_{0}\rangle^{-\frac{3}{2}},
\end{eqnarray}
where the constant $C$ is independent of $t, t_{0}$.

\emph{Step 2, local decay estimates.} For $\psi_{0}\in \mathscr{A}(t)$ and $\sigma>\frac{3}{2}$
\begin{eqnarray}\label{eq-10-2}
\big\|\langle x-D(t)\rangle^{-\sigma}U(t)\psi_{0}\big\|_{L_{t}^{2}L_{x}^{2}}\leq C\|\psi_{0}\|_{L^2}.
\end{eqnarray}

\emph{Step 3, local decay estimates for source part}.  We shall show that for  $\sigma>\frac{3}{2}$ and admissible pair $(\tilde{p},\tilde{q})$
\begin{eqnarray}\label{eq-15}
\Big\|\langle x-D(t)\rangle^{-\sigma}\int_{0}^{t}U(t,s)P_{c}(s)F(s)ds\Big\|_{L^{2}_{t}L^{2}_{x}}\lesssim \|\psi_{0}\|_{L^{2}}+\|F\|_{L^{\tilde{p}'}_{t}L^{\tilde{q}'}_{x}}.
\end{eqnarray}
Consider the Cauchy problem
\begin{eqnarray}\label{eq-16}
&&i\Phi_{t}=-\frac{1}{2}\Delta\Phi+F(t)=(-\frac{1}{2}\Delta+V_{1}+V_{2}(x-\vec{e}_{1}t))\Phi-(V_{1}+V_{2}(x-\vec{e}_{1}t))\Phi+F(t)\nonumber\\
&&\Phi(0,\cdot)=\psi_{0}\in \mathscr{A}(0).\nonumber
\end{eqnarray}
It follows from Duhamel's formula that
\begin{eqnarray}\label{eq-11-3}
P_{c}(t)\Phi(t)=P_{c}(t)U(t)\psi_{0}&+&i\int^{t}_{0}P_{c}(t)U(t,s)(V_{1}+V_{2}(\cdot-\vec{e}_{1}s))\Phi(t)ds\nonumber\\
&&\ \ \ \ \ \ \ \ \ \ \ \ \ \ -\ i\int^{t}_{0}P_{c}(t)U(t,s)F(s)ds.
\end{eqnarray}
 For the left hand side of (\ref{eq-11-3}), since $\Phi(t)$ is also a solution of $i\Phi_{t}=-\frac{1}{2}\Delta\Phi+F(t)$, by (\ref{s-f'}), (\ref{eq-12-1}) and  Duhamel's formula again, we have that for any  $\sigma_{1}>\frac{3}{2}$ and admissible pair $(\tilde{p},\tilde{q})$
\begin{eqnarray}\label{eq-17}
\|\langle x-D(t)\rangle^{-\sigma_{1}}\Phi(t)\|_{L^{2}_{t}L^{2}_{x}}
&\leq& \big\|\langle x-D(t)\rangle^{-\sigma_{1}}e^{it\frac{\Delta}{2}}\psi_{0}\big\|_{L^{2}_{t}L^{2}_{x}}+
\Big\|\langle x-D(t)\rangle^{-\sigma_{1}}\int_{0}^{t}e^{i(t-s)\frac{\Delta}{2}}F(s)ds\Big\|_{L^{2}_{t}L^{2}_{x}}\nonumber\\
&\leq& C\Big(\|\psi_{0}\|_{L^{2}}+\Big\|\int_{0}^{t}e^{i(t-s)\frac{\Delta}{2}}F(s)ds\Big\|_{L^{2}_{t}L^{6}_{x}}\Big)\nonumber\\
&\leq& C\Big(\|\psi_{0}\|_{L^{2}}+\|F\|_{L^{\tilde{p}'}_{t}L^{\tilde{q}'}_{x}}\Big),
\end{eqnarray}
which combining with Proposition \ref{lem-1} leads to
\begin{eqnarray}\label{eq-18}
&&\|\langle x-D(t)\rangle^{-\sigma}P_{c}(t)\Phi(t)\|_{L^{2}_{t}L^{2}_{x}}\nonumber\\
&\leq& \big\|\langle x-D(t)\rangle^{-\sigma}\Phi(t)\big\|_{L_{t}^{2}L_{x}^{2}}+\sum^{2}_{\kappa=1}\big\|\langle x-D(t)\rangle^{-\sigma}P_{\kappa b}(t)\Phi(t)\big\|_{L_{t}^{2}L_{x}^{2}}\nonumber\\
&\leq& \big\|\langle x-D(t)\rangle^{-\sigma}\Phi(t)\big\|_{L_{t}^{2}L_{x}^{2}}+\sup_{t}\big\|\langle x-D(t)\rangle^{-\sigma}P_{1b}(t)\langle x\rangle^{\sigma_{1}}\big\|_{L_{x}^{2}\rightarrow L_{x}^{2}}\big\|\langle x\rangle^{-\sigma_{1}}\Phi(t)\big\|_{L_{t}^{2}L_{x}^{2}}\nonumber\\
&&\ \ \ \ \ \ \ \ \ \ \ \ \ \ +\sup_{t}\big\|\langle x-D(t)\rangle^{-\sigma}P_{2b}(t)\langle x-\vec{e}_{1}t\rangle^{\sigma_{1}}\big\|_{L_{x}^{2}\rightarrow L_{x}^{2}}\big\|\langle x-\vec{e}_{1}t\rangle^{-\sigma_{1}}\Phi(t)\big\|_{L_{t}^{2}L_{x}^{2}}\nonumber\\
&\leq&C\|\psi_{0}\|_{L^{2}}+C\|F\|_{L^{\tilde{p}'}_{t}L^{\tilde{q}'}_{x}}.
\end{eqnarray}
The local decay estimate for the first term of the right side of (\ref{eq-11-3}) follows from (\ref{eq-10-2}), and
the local decay estimate for the second term of the right side of (\ref{eq-11-3}) follows
\begin{eqnarray}\label{eq-19}
&&\Big\|\int^{t}_{0}\langle x-D(t)\rangle^{-\sigma}P_{c}(t)U(t,s)(V_{1}+V_{2}(\cdot-\vec{e}_{1}s))\Phi(s)ds\Big\|_{L_{t}^{2}L_{x}^{2}}\nonumber\\
&\leq& C\Big\|\int^{t}_{0}\langle t-s\rangle^{-\frac{3}{2}}\big(\big\||V_{1}|^{\frac{1}{2}}\Phi(s)\big\|_{L_{x}^{2}}+\big\||V_{2}(\cdot-\vec{e}_{1}s)|^{\frac{1}{2}}\Phi(s)\big\|_{L_{x}^{2}}\big)ds\Big\|_{L_{t}^{2}}\nonumber\\
&\leq& C\Big(\big\||V_{1}|^{\frac{1}{2}}\langle x\rangle^{\sigma}\langle x\rangle^{-\sigma}\Phi(t)\big\|_{L^{2}_{t}L_{x}^{2}}+\big\||V_{2}(\cdot-\vec{e}_{1}t)|^{\frac{1}{2}}\langle x-\vec{e}_{1}t\rangle^{\sigma}\langle x-\vec{e}_{1}t\rangle^{-\sigma}\Phi(t)\big\|_{L^{2}_{t}L_{x}^{2}}\Big)\nonumber\\
&\leq& C\|\psi_{0}\|_{L^{2}}+C\|F\|_{L^{\tilde{p}'}_{t}L^{\tilde{q}'}_{x}},
\end{eqnarray}
where we use Kato-Jensen estimates, Young's inequality and (\ref{eq-17}). Thus the local decay estimate for the source term (\ref{eq-15}) follows from the (\ref{eq-11-3})-(\ref{eq-19}).

\emph{Step 4, local decay estimates for the solution of (\ref{eq-14'}).} The solution of (\ref{eq-14'}) can be presented by Duhamel's formula as follows,
\begin{eqnarray}
P_{c}(t)\psi(t)=U(t)\psi_{0}-i\int_{0}^{t}U(t,s)P_{c}(s)F(s)ds.\nonumber
\end{eqnarray}
Then by Steps 2 and 3, for $\sigma>\frac{3}{2}$ and admissible pair $(\tilde{p},\tilde{q})$, we have
\begin{eqnarray}\label{eq-29}
\|\langle x-D(t)\rangle^{-\sigma}P_{c}(t)\psi(t)\|_{L_{t}^{2}L_{x}^{2}}&\leq& \|\langle x-D(t)\rangle^{-\sigma}U(t)\psi_{0}\|_{L_{t}^{2}L_{x}^{2}}\nonumber\\
&&\ \ \ \ \ \ \ \ \ \ \ +\
\Big\|\int_{0}^{t}\langle x-D(t)\rangle^{-\sigma}U(t,s)P_{c}(s)F(s)ds\Big\|_{L_{t}^{2}L_{x}^{2}}\nonumber\\
&\leq& C\|\psi_{0}\|_{L^{2}}+C\|F\|_{L^{\tilde{p}'}_{t}L^{\tilde{q}'}_{x}}.
\end{eqnarray}

\emph{Step 5, the Strichartz estimates.} We write equation (\ref{eq-14'}) as
\begin{eqnarray}
i\partial_{t}(P_{c}(t)\psi)=-\frac{1}{2}\Delta P_{c}(t)\psi+(V_{1}+V_{2}(x-\vec{e}_{1}t))P_{c}(t)\psi+P_{c}(t)F(t).\nonumber
\end{eqnarray}
The Duhamel formula leads to
\begin{eqnarray}\label{eq-20}
\|P_{c}(t)\psi(t)\|_{L_{t}^{p}L_{x}^{q}}&\leq& \big\|e^{it\frac{\Delta}{2}}\psi_{0}\big\|_{L^{p}_{t}L_{x}^{q}}+\Big\|\int_{0}^{t}e^{i(t-s)\frac{\Delta}{2}}P_{c}(s)F(s)ds\Big\|_{L_{t}^{p}L_{x}^{q}}\nonumber\\
&&\ \ \ \ \ \ \ \ +\ \Big\|\int_{0}^{t}e^{i(t-s)\frac{\Delta}{2}}(V_{1}+V_{2}(\cdot-\vec{e}_{1}s))P_{c}(s)\psi(s)ds\Big\|_{L_{t}^{p}L_{x}^{q}}\nonumber\\
&\leq& C\|\psi_{0}\|_{L^{2}}+\|P_{c}(t)F(t)\|_{L^{\tilde{p}'}_{t}L^{\tilde{q}'}_{x}}+\|(V_{1}+V_{2}(\cdot-\vec{e}_{1}t))P_{c}(t)\psi(t)\|_{L_{t}^{2}L_{x}^{\frac{6}{5}}}\nonumber\\
&\leq& C\Big ( \|\psi_{0}\|_{L^{2}}+\|F(t)\|_{L^{\tilde{p}'}_{t}L^{\tilde{q}'}_{x}}\Big),
\end{eqnarray}
where in the last inequality we use the $L^{\tilde{q}'}$ boundedness of $P_{c}(t)$ and the fact (follows from (\ref{eq-29})) that
\begin{eqnarray}
&&\|(V_{1}+V_{2}(\cdot-\vec{e}_{1}t))P_{c}(t)\psi(t)\|_{L_{t}^{2}L_{x}^{\frac{6}{5}}}\nonumber\\
&\leq& C\big\||V_{1}|^{\frac{1}{2}}\langle x\rangle^{\sigma}\langle x\rangle^{-\sigma}P_{c}(t)\psi(t)\big\|_{L^{2}_{t}L_{x}^{2}}
+C\big\||V_{2}(\cdot-\vec{e}_{1}t)|^{\frac{1}{2}}\langle x-\vec{e}_{1}t\rangle^{\sigma}\langle x-\vec{e}_{1}t\rangle^{-\sigma}P_{c}(t)\psi(t)\big\|_{L^{2}_{t}L_{x}^{2}}\nonumber\\
&\leq& C\Big(\|\psi_{0}\|_{L^{2}}+\|F(t)\|_{L^{\tilde{p}'}_{t}L^{\tilde{q}'}_{x}}\Big).\nonumber
\end{eqnarray}
Hence we finish the proof.

\section{The Strichartz estimates for matrix charge transfer model}
\setcounter{equation}{0}

\subsection{Notations and assumptions}
In this section, the endpoint Strichartz estimates for charge transfer model with matrix potentials will be considered. Without loss of generality, we assume that the number of potentials in (\ref{MSch eq}) is $m=2$ and
that the velocities are $\vec{v}_{1}=0$, $\vec{v}_{2}=(1,0,\ldots,0)=\vec{e}_{1}$. Thus we turn to study  the problem
\begin{eqnarray}
\left\{%
\begin{array}{ll}\label{main-equation-2}
i\partial_{t}\vec{\psi}=\mathscr{H}_{0}\vec{\psi}+V_{1}(t,x)\vec{\psi}+V_{2}(t,x-\vec{e}_{1}t)\vec{\psi}:=\mathscr{H}(t)\vec{\psi},\\
\vec{\psi}(0,\cdot)=\vec{\psi}_{0},\ \ x\in \mathbb{R}^{n},\\
\end{array}
\right.
\end{eqnarray}
 where
$$
\mathscr{H}_{0}=\left(
  \begin{array}{cc}
    -\frac{1}{2}\Delta & 0 \\
    0 & \frac{1}{2}\Delta \\
  \end{array}
\right),\
V_{\kappa}(t,x)=\left(
  \begin{array}{cc}
    U_{\kappa}(x) & -e^{i\theta_{\kappa}(t,x)}W_{\kappa}(x) \\
    e^{-i\theta_{\kappa}(t,x)}W_{\kappa}(x) & -U_{\kappa}(x) \\
  \end{array}
  \right), \ \ \kappa=1,2,
$$
with
\begin{eqnarray}\label{eq-35}
\theta_{1}(t,x)=\alpha^{2}_{1}t+\gamma_{1} \ \ {\rm and}\ \  \theta_{2}(t,x)=(|\vec{e}_{1}|^{2}+\alpha^{2}_{2})t+2x\cdot \vec{e}_{1}+\gamma_{2}.
\end{eqnarray}

 As seen in Section 2, Galilei transform plays important roles in the scalar case, here we will need the  vector-valued Galilei transform which is defined as follows:
\begin{eqnarray}\label{eq-33}
\mathcal {G}_{\vec{v},y}(t)\left(
  \begin{array}{cc}
    \psi_{1}  \\
    \psi_{2}  \\
  \end{array}
  \right):=\left(
  \begin{array}{cc}
    G_{\vec{v},y}(t)\psi_{1}  \\
    \overline{G_{\vec{v},y}(t)\overline{\psi_{2}}}\\
  \end{array}
  \right)
\end{eqnarray}
where $G_{\vec{v},y}(t)$ are defined by (\ref{Ga-1}). It is easy to see that the transformations $\mathcal {G}_{\vec{v},y}(t)$ are isometries on all $L^{p}$ spaces. We set $\mathcal {G}_{\vec{v}}(t)=\mathcal {G}_{\vec{v},0}(t)$ for $y=0$, and then $\mathcal {G}_{\vec{v}}(t)^{-1}=\mathcal {G}_{-\vec{v}}(t)$.
In contrast to the scalar case, a modulation transform
\begin{eqnarray}\label{eq-36}
\mathcal {M}(t)=\mathcal {M}_{\alpha,\gamma}(t)=\left(
  \begin{array}{cc}
    e^{-\frac{i}{2}\omega(t)}& 0 \\
    0 & e^{\frac{i}{2}\omega(t)} \\
  \end{array}
\right)
\end{eqnarray}
with $\omega(t)=\alpha^{2}t+\gamma$ will be involved, which combing the Galilei transform can reduce the Schr\"{o}dinger operator with one time-dependent matrix potential ($m=1$ in (\ref{MSch eq})) to time-independent case.

%Moreover, we require that each $\mathscr{H}_{\kappa}, \ \kappa=1,2$ be admissible in the sense of Definition \ref{defi-4}.

To study the endpoint Strichartz estimates for matrix charge transfer model, we would impose certain  assumptions on time-independent Hamiltonians:
\begin{eqnarray}\label{eq-34}
\mathscr{H}_{\kappa}=\left(
  \begin{array}{cc}
    -\frac{1}{2}\Delta+\frac{1}{2}\alpha^{2}_{\kappa}+U_{\kappa} & -W_{\kappa} \\
    W_{\kappa} &  \frac{1}{2}\Delta-\frac{1}{2}\alpha^{2}_{\kappa}-U_{\kappa} \\
  \end{array}
  \right), \ \ \ \ 1\leq\kappa\leq 2.
\end{eqnarray}
In fact, motivated by applications to NLS, the authors of \cite{RSS1} require that $\mathscr{H}_{\kappa}$ to be admissible and satisfy stability condition.
For this end, following \cite[Definition 7.2]{RSS1} we will introduce these conditions in general setting. Set
 $H=-\frac{1}{2}\Delta+\mu$ where $\mu>0$ and
$$
B=\left(
  \begin{array}{cc}
    H & 0 \\
    0 & -H \\
  \end{array}
  \right),\ \ \ \
 V=\left(
  \begin{array}{cc}
    U & -W \\
    W & -U \\
  \end{array}
  \right),\ \
  A=B+V=\left(
  \begin{array}{cc}
    H+U & -W \\
    W & -H-U \\
  \end{array}
  \right),
$$
with $U,\ W$ real-valued. We remark that such non-selfadjoint operators arise when linearizing a focusing NLS equation around one soliton and then  studying the nonlinear asymptotic stability of one soliton.  One can also find many  spectra results on the special non-selfadjoint operator in \cite{RSS1, RSS2} and the exponentially decaying properties of generalized  eigenfunctions of such operators in \cite{HL}, which in parts motivate the following spectral assumptions on A.

\begin{defi}\label{defi-4}
Let $A=B+V$ defined as above. We call the operator A on $\mathcal {H}=L^{2}(\mathbb{R}^{n})\times L^{2}(\mathbb{R}^{n})$ {\rm admissible} provided that

{\rm (I)} ${\rm spec(A)\subset\mathbb{R}}$ and ${\rm spec(A)\cap (-\mu,\mu)=\{\omega_{\ell}|0\leq\ell\leq M\}}$, where $\omega_{0}=0$ and all $\omega_{\ell}$ are distinct eigenvalues. There are no eigenvalues in {\rm $spec_{ess}(A)$}. Furthermore, the points $\pm\mu$ are not resonances of $A$.

{\rm (II)} For $1\leq \ell\leq M$,  $L_{\ell}{\rm :=ker(A-\omega_{\ell})^{2}=ker(A-\omega_{\ell})}$, and  ${\rm ker(A)\subsetneq ker(A^{2})=ker(A^{3})=:}L_{0}$.
        Moreover, these spaces are finite dimensional.

{\rm (III)} The ranges ${\rm Ran(A-\omega_{\ell})}$ for $1\leq \ell\leq M$ and ${\rm Ran(A^{2})}$ are closed.

{\rm (IV)} The spaces $L_{\ell}$ are spanned by exponentially decreasing functions in $\mathcal {H}$ (say with bound $e^{-\varepsilon_{0}|x|}$).

{\rm (V)} All these assumptions hold as well for the adjoint $A^{\ast}$. We denote the corresponding (generalized) eigenspaces by $L^{\ast}_{\ell}$.
\end{defi}

For each admissible operator $A$ on $\mathcal {H}$ defined above,  we have  the following useful lemma due to \cite[Lemma 7.3]{RSS1}.
\begin{lem}\label{lem-2}
There is a direct sum decomposition
\begin{eqnarray}\label{eq-30}
\mathcal {H}=\sum_{j=0}^{M}L_{j}+\Big(\sum_{j=0}^{M}L_{j}^{\ast}\Big)^{\perp},
\end{eqnarray}
which is invariant under $A$. Let $P_{c}$ denote the projection onto $\Big(\sum_{j=0}^{M}L_{j}^{\ast}\Big)^{\perp}$ which is induced by (\ref{eq-30}) and set $P_{b}=I-P_{c}$. Then $AP_{c}=P_{c}A$ and there exist number $c_{ij}$ such that
\begin{eqnarray}\label{eq-31}
P_{b}f=\sum_{i,j}c_{ij}\phi_{j}\langle \psi_{i},f\rangle.
\end{eqnarray}
where $\phi_{j},\ \psi_{i}$ are exponentially decreasing functions.
\end{lem}

We further make additional assumption for $A$, namely the stability assumption,
\begin{eqnarray}\label{eq-32}
\sup_{t\in \mathbb {R}}\big\|e^{itA}P_{c}\big\|<\infty,
\end{eqnarray}
where $\|\cdot\|$ refers to the operator norm on $\mathcal {H}$ without further description and $P_{c}$ is the projection  defined as in Lemma \ref{lem-2}.
The stability assumption (\ref{eq-32}) says that the semigroup is uniformly bounded in $t$ for all functions in $\Big(\sum_{j=0}^{M}L_{j}^{\ast}\Big)^{\perp}$, which is related to  the notion of {\it linear stability } in the context of stability
of soliton solution of NLS; see, e.g., \cite{Wein2}. Furthermore, if $A$ is admissible in the sense of Definition \ref{defi-4} and satisfies (\ref{eq-32}),  we know that the semigroup is actually uniformly bounded on all functions that have no component in the space $L_{0}$. Otherwise, it can grow at most like $t$.
%Now let us go back to the matrix charge transfer model,
%We know that for Schr\"{o}dinger equation (scalar or matrix) with potentials, the initial data should lies in certain space to make the solution dispersive. To define such space for matrix case,

\begin{rem}{\rm If the functions $U$ and $W$ in matrix $V$  are $-\Delta$ bounded with relative bound zero and decay to zero at infinity, the authors in \cite{HL} identified the essential spectrum and proved the assumption (IV) is satisfied, see also \cite{RSS2}. If the matrix operator $A$ comes from the linearization of NLS, then one can find further results on eigenvalues, generalized eigenspaces and stability condition, see e.g. \cite{Cuc1, ES, RSS2, Wein2}.}
\end{rem}

\vskip0.3cm

Let us turn to the matrix charge transfer model (\ref{main-equation-2}). Since the projections $P_{c}(\mathscr{H}_{\kappa})$ and $P_{b}(\mathscr{H}_{\kappa})$ as in Lemma  \ref{lem-2}  are no longer orthogonal,
in \cite{RSS1}, the authors introduce a notion of ``scattering states" instead of ``asymptotically orthogonal to the bound states" in scalar case,  and proved that if one choose a initial data in ``scattering states'', the decay estimates holds. The ``scattering states'' for matrix case are defined as follows.
\begin{defi}\label{def-3}
Assume that  $\mathscr{H}_{\kappa} \ (\kappa=1,2)$ is admissible in the sense of Definition \ref{defi-4} and satisfies the stability condition (\ref{eq-32}). Let $\mathcal {U}(t,s)$ be the propagator of the equation (\ref{main-equation-2}). Then we say that $\vec{f}$ is a scattering state  if $\vec{f}$ belongs to
  \begin{eqnarray*}
\mathscr{A}(s)=\big\{\vec{f}\in \mathcal {H}:\ \|P_{b}(\mathscr{H}_{1},t)\mathcal {U}(t,s)\vec{f}\|_{L^2}+\|P_{b}(\mathscr{H}_{2},t)\mathcal {U}(t,s)\vec{f}\|_{L^2}\rightarrow0\ \ {\rm as} \ \ t\rightarrow +\infty\big\},
\end{eqnarray*}
where
\begin{eqnarray}
P_{b}(\mathscr{H}_{1},t)=\mathcal {M}_{\alpha_{1}, \gamma_{1}}(t)^{-1}P_{b}(\mathscr{H}_{1})\mathcal {M}_{\alpha_{1}, \gamma_{1}}(t)
\end{eqnarray}
\begin{eqnarray}
P_{b}(\mathscr{H}_{2},t)=\mathcal {G}_{\vec{e}_{1}}(t)^{-1}\mathcal {M}_{\alpha_{2}, \gamma_{2}}(t)^{-1}P_{b}(\mathscr{H}_{2})\mathcal {M}_{\alpha_{2}, \gamma_{2}}(t)\mathcal {G}_{\vec{e}_{1}}(t)
\end{eqnarray}
with $\mathcal {G}_{\vec{e}_{1}}(t)$ and $\mathcal {M}_{\alpha, \gamma}(t)$ defined by (\ref{eq-33}) and (\ref{eq-36}), respectively.
\end{defi}

\begin{rem}\label{rem-1'} {\rm
 (i) It is proved in \cite[Proposition 8.5]{RSS1} that if $\vec{f}\in \mathscr {A}(s)$ for given $s$, then
\begin{eqnarray}
\|P_{b}(\mathscr{H}_{1},t)\mathcal {U}(t,s)\vec{f}\|_{L^2}+\|P_{b}(\mathscr{H}_{2},t)\mathcal {U}(t,s)\vec{f}\|_{L^2}\lesssim e^{-\alpha t}\|\vec{f}\|_{L^{2}}\nonumber
\end{eqnarray}
for some $\alpha>0$.

(ii) The authors in \cite[Theorem 8.4]{RSS1} proved that if the initial data $\psi_{0}\in \mathscr {A}(0)\cap L^{1}$,
\begin{eqnarray}
\big\|\mathcal {U}(t)\psi_{0}\big\|_{L^{2}+L^{\infty}}\lesssim \langle t\rangle^{-\frac{3}{2}}\|\psi_{0}\|_{L^{1}\cap L^{2}}.\nonumber
\end{eqnarray}
Moreover, one can remove $L^2$ on the left hand side if in addition the matrix potentials $V_{\kappa}(t,x)$ satisfy  $\sup_t\|\widehat{V}_{\kappa}(t,\cdot)\|_{L^{1}}<\infty$ for $\kappa=1,2$.
}
\end{rem}

\subsection{Wave operators and asymptotic completeness}

As in the scalar case, the asymptotic completeness has been taken into account to prove the Strichartz estimates. Let $\mathcal {U}(t)$ be the propagator of the equation (\ref{main-equation-2}) and $P_{c}(s)$ is the projection onto $\mathscr {A}(s)$, similarly to the scalar case, one can easily check from Definition \ref{def-3} that
\begin{eqnarray}\label{ex-eq-2}
\mathcal{U}(t,s)\mathscr{A}(s)=\mathscr{A}(t).
\end{eqnarray}
 Besides,  we need to make an additional assumption on $\mathcal {U}(t)$.

\textbf{Growth Assumption:} \emph{There exists a constant $m_{0}>0$ such that}
\begin{eqnarray}\label{ap}
\|\mathcal {U}(t)\|_{L^{2}\rightarrow L^{2}}\lesssim \langle t\rangle^{m_{0}}.
\end{eqnarray}

\begin{rem}{\rm
(i) We know that $\mathcal {U}(t)$ is no longer unitary since $\mathscr{H}(t)$ is not a self-adjoint operator. On the other hand,  even if $\mathscr{H}$ is a admissible matrix Schr\"{o}dinger operator  with only one time-independent potential ( in sense of Definition \ref{defi-4} and (\ref{eq-32})) , $e^{it\mathscr{H}}$ may grow like $t$ in $L^{2}$ ( see e.g. \cite{RSS1} ).

(ii) Under the assumption (\ref{ap}), one can use the same method as in the proof of Proposition 2 in \cite{Bou} to show the Sobolev estimates for $\mathcal {U}(t)$
\begin{eqnarray}\label{se}
\|\mathcal {U}(t)\|_{H^{s}\rightarrow H^{s}}\lesssim \langle t\rangle^{m_{s}},
\end{eqnarray}
where $H^{s}$ ($s>0$) is the Sobolev space and $m_{s}>0$ depends on $m_{0}$.
}

\end{rem}

Now we first define wave operators for matrix case,
%\begin{eqnarray}
%\Omega^{-}_{0}(s)=s-\lim_{t\rightarrow+\infty}\mathcal{U}(s,t)P_{c}(t)\mathcal{U}_{0}(t-s),\nonumber
%\end{eqnarray}
\begin{eqnarray}
\Omega^{-}_{1}(s)=s-\lim_{t\rightarrow+\infty}\mathcal{U}(s,t)\mathcal {M}_{\alpha_{1}, \gamma_{1}}(t)^{-1}\mathcal{U}_{1}(t,s)P_{b}(\mathscr{H}_{1})\mathcal {M}_{\alpha_{1}, \gamma_{1}}(s),\nonumber
\end{eqnarray}
\begin{eqnarray}
\Omega^{-}_{2}(s)=s-\lim_{t\rightarrow+\infty}\mathcal{U}(s,t)\mathcal {G}_{-\vec{e_{1}}}(t)\mathcal {M}_{\alpha_{2}, \gamma_{2}}(t)^{-1}\mathcal{U}_{2}(t,s)P_{b}(\mathscr{H}_{2})\mathcal {M}_{\alpha_{2}, \gamma_{2}}(s)\mathcal {G}_{\vec{e_{1}}}(s).\nonumber
\end{eqnarray}
We investigate the range of $\Omega^{-}_{1}(s)$ for each $s$. Notice that
\begin{eqnarray}\label{commut-1}
\mathcal {U}(s,t)\mathcal {M}_{\alpha_{1}, \gamma_{1}}(t)^{-1}=
\mathcal {M}_{\alpha_{1}, \gamma_{1}}(s)^{-1}\widetilde{\mathcal {U}}(s,t).
\end{eqnarray}
We then turn to consider
\begin{eqnarray}
\widetilde{\Omega}^{-}_{1}(s)=s-\lim_{t\rightarrow+\infty}\widetilde{\mathcal{U}}(s,t)\mathcal{U}_{1}(t,s)P_{b}(\mathscr{H}_{1}).\nonumber
\end{eqnarray}
Here $\widetilde{\mathcal {U}}(s,t)$ is the propagator of $i\partial_{t}\vec{\psi}=\widetilde{\mathscr{H}}(t)\vec{\psi}$ with
\begin{eqnarray}\label{eq-38}
\widetilde{\mathscr{H}}(t)=\mathscr{H}_{1}+\widetilde{V}_{2}(t,x-\vec{e}_{1}t)
\end{eqnarray}
where $\mathscr{H}_{1}$ is defined by (\ref{eq-34}),
\begin{eqnarray}
\widetilde{V}_{2}(t,x)=\left(
  \begin{array}{cc}
    U_{2}(x) & -e^{i(\theta_{2}-\theta_{1})(t,x)}W_{2}(x) \\
    e^{-i(\theta_{2}-\theta_{1})(t,x)}W_{2}(x) & -U_{2}(x) \\
  \end{array}
  \right), \nonumber
\end{eqnarray}
and $\theta_{1}, \theta_{2}$ are defined by (\ref{eq-35}), and it follows from the identity (\ref{commut-1}) that $\widetilde{\mathcal {U}}(t)$ also has the same growth as (\ref{ap}) and (\ref{se}).

\begin{lem}\label{thm-4} Assume that the assumption (\ref{ap}) is satisfied,  $\mathscr{H}_{\kappa} \ (\kappa=1,2)$ is admissible in the sense of Definition \ref{defi-4} and satisfies the stability condition (\ref{eq-32}).
Let $\vec{u}$ be a generalized eigenfunction of $\mathscr{H}_{1}$. Then for $\sigma\geq0$ and multi-index $\gamma$ with $|\gamma|\geq0$,
$$\big\|\langle x\rangle^{\sigma}\partial^{\gamma}_{x}\widetilde{\Omega}^{-}_{1}(s)\vec{u}\big\|_{L^{2}}$$
 is uniformly bounded in $s$.
\end{lem}

\begin{proof}
We will use similar approach as in the proof of Theorem \ref{thm0}. Notice that it follows from the Duhamel formula that
%both $iH_{1}$ and $-iH_{1}$ generate semigroups, then we can apply spectral theorem to see
%$\mathcal{U}_{1}(r,s)\vec{u}=e^{-i\lambda(r-s)}\vec{u}$  if $\lambda$ is not zero.
%However, if $\lambda=0$, then $u\in Ker(H^{2}_{1})$, it also follows that
\begin{eqnarray}\label{eq-38}
\widetilde{\mathcal{U}}(s,t)\mathcal{U}_{1}(t,s)P_{b}(\mathscr{H}_{1})\vec{u}&=&\vec{u}+i\int^{t}_{s}\widetilde{\mathcal{U}}(s,r)\widetilde{V}_{2}(r,\cdot-\vec{e}_{1}r)\mathcal{U}_{1}(r,s)\vec{u}dr\nonumber\\
&=&\vec{u}+i\int^{t}_{s}\widetilde{\mathcal{U}}(s,r)\widetilde{V}_{2}(r,\cdot-\vec{e}_{1}r)e^{-i\omega(r-s)}\vec{u}dr
\end{eqnarray}
for some $\omega\neq0$. If $\vec{u}\in {\rm Ker}(\mathscr{H}^{2}_{1})$, then in the integration of (\ref{eq-38}), $\omega=0$ and  $\vec{u}$
will be replaced by $(r-s)\vec{u}$, which would not affect the proof.
%we know that for $\vec{u}\in {\rm Ker}(H^{2}_{1})$, $\mathcal {H_{1}}(r,s)\vec{u}$ grow at most $r-s$,
%and since $s<r$, we still have that $$\big\|\langle x\rangle^{\sigma}\partial^{\beta}_{x}\widetilde{V_{2}}%\mathcal {H_{1}}(r,s)\vec{u}\big\|_{L_{x}^{2}}\leq c(\sigma,|\beta|,N)\langle r\rangle^{-N}.$$

Since $V_{1}$ and $V_{2}$ are exponentially localized  and smooth, $\vec{u}$ is also smooth and exponentially localized in $L^{2}$, thus $\langle x\rangle^{\sigma}\partial^{\gamma}_{x}\vec{u}\in L^{2}$ for all $\sigma>0$ and multi-index $\gamma$.  On the other hand, the vector
$$\vec{K}(r,s,x):=\widetilde{V}_{2}(x-\vec{e}_{1}r)e^{-i\omega(r-s)}\vec{u}(x)$$
has the property that for any $\sigma>0$, multi-index $\beta$ and $N>0$
$$\big\|\langle x\rangle^{\sigma}\partial^{\beta}_{x}\vec{K}(r,s,\cdot)\big\|_{L_{x}^{2}}\leq c(\sigma,|\beta|,N)\langle r\rangle^{-N}.$$
Thus even $\widetilde{U}(t)$ is allowed to have certain growth polynomially, it follows from the above inequality that (\ref{eq-38}) is well-defined. Moreover, for any $j\geq0$ and any multi-index $\gamma$, define
\begin{eqnarray}\label{eq-39}
\Phi_{j,|\gamma|}(t)
=\sum^{j}_{j'=0}\sum_{|\gamma'|=0}^{\gamma}\big\| \langle x\rangle^{j'} \partial^{\gamma'}_{x}\widetilde{\mathcal{U}}(t)\vec{g}\big\|_{L^{2}}.\nonumber
\end{eqnarray}
Notice that
\begin{eqnarray}\label{eq-40}
\frac{d}{dt}\big\|\langle x\rangle^{j'} \partial^{\gamma'}_{x}\widetilde{\mathcal{U}}(t)\vec{g}\big\|_{L^{2}}&=&-i\Big\langle \big[\langle x\rangle^{j'} \partial^{\gamma'}_{x}, \widetilde{\mathscr{H}}(t)\big]\widetilde{\mathcal{U}}(t)\vec{g}, \langle x\rangle^{j'} \partial^{\gamma'}_{x}\widetilde{\mathcal{U}}(t)\vec{g}\Big\rangle\nonumber\\
&&\ \ \ -i\Big\langle \langle x\rangle^{j'} \partial^{\gamma'}_{x}\widetilde{\mathcal{U}}(t)\vec{g},\big[\widetilde{\mathscr{H}}^{\ast}(t), \langle x\rangle^{j'} \partial^{\gamma'}_{x}\big]\widetilde{\mathcal{U}}(t)\vec{g}\Big\rangle\nonumber\\
&&\ \ \ \ +i\Big\langle \langle x\rangle^{j'} \partial^{\gamma'}_{x}\widetilde{\mathcal{U}}(t)\vec{g},\langle x\rangle^{j'} \partial^{\gamma'}_{x}\big(\widetilde{\mathscr{H}}(t)-\widetilde{\mathscr{H}}^{\ast}(t)\big)\widetilde{\mathcal{U}}(t)\vec{g}\Big\rangle
\end{eqnarray}
For first two terms are of same type, we only consider
\begin{eqnarray}\label{eq-40'}
\Big|\big\langle \big[\langle x\rangle^{j'} \partial^{\gamma'}_{x}, \widetilde{H}(t)\big]\widetilde{\mathcal{U}}(t)\vec{g}, \langle x\rangle^{j'} \partial^{\gamma'}_{x}\widetilde{\mathcal{U}}(t)\vec{g}\big\rangle\Big| &\lesssim& \Phi_{j',|\gamma'|}(0)+\langle t\rangle^{2}\sup_{0\leq \tau\leq t}\Phi_{j'-1,|\gamma'|+1}(\tau)\nonumber\\
&\lesssim& \sum^{j'-1}_{k=0}\Phi_{j'-k,|\gamma'|+k}(0)+\langle t\rangle^{2j'}\Phi_{0,|\gamma'|+j'}(\tau)\nonumber\\
&\lesssim& \langle t\rangle^{K}\sum_{|\beta|<|\gamma'|+j}\big\|\langle x\rangle^{j} \partial^{\beta}_{x}\vec{g}\big\|_{L^{2}}
\end{eqnarray}
where we use the Sobolev estimates (\ref{se}) for $\Phi_{0,|\gamma'|+j'}(t)$ and $K>0$ depends on $\gamma$, $j$ and $m_{s}$ in (\ref{se}). For the last term of (\ref{eq-40}), notice that $\widetilde{\mathscr{H}}(t)-\widetilde{\mathscr{H}}^{\ast}(t)$ is the matrix with localized off-diagonal terms, it would also be bounded by (\ref{eq-40'}).
Thus we have
\begin{eqnarray}
\Phi_{j,|\gamma|}(t)\lesssim \langle t\rangle^{K}\sum_{|\beta|<|\gamma|+j}\big\|\langle x\rangle^{j} \partial^{\beta}_{x}\vec{g}\big\|_{L^{2}},\nonumber
\end{eqnarray}
and combining with the idea of the proof of Lemma \ref{thm0} imply the desired conclusion.
\end{proof}

\begin{rem}{\rm
By the structure of $P_{b}$ in Lemma \ref{lem-2},  choosing suitable $\sigma$ and $\gamma$ in Theorem \ref{thm-4} and the Sololev embedding, we could obtain that the wave operator  $\Omega^{-}_{1}(s)$ is uniformly  bounded on $L^{p}$ for all $1\le p\le \infty$.
}
\end{rem}

\begin{thm}\label{thm-5}Assume that the assumption (\ref{ap}) is satisfied,  $\mathscr{H}_{\kappa} \ (\kappa=1,2)$ is admissible in the sense of Definition \ref{defi-4} and satisfies the stability condition (\ref{eq-32}). Then for every $s>0$, there is a direct sum decomposition
\begin{eqnarray}\label{ac-2}
\mathcal {H}=\mathscr {A}(s)\oplus{\rm Ran \Omega^{-}_{1}(s)\oplus Ran \Omega^{-}_{2}(s)}.
\end{eqnarray}\nonumber
\end{thm}

\begin{proof}For any $\vec{\phi}_{0}\in \mathcal {H}$,  write
\begin{eqnarray}
\vec{\phi}(t,s):=\mathcal {U}(t,s)\vec{\phi}_{0}=P_{b}(\mathscr{H}_{1},t)\mathcal {U}(t,s)\vec{\phi}_{0}+P_{b}(\mathscr{H}_{2},t)\mathcal {U}(t,s)\vec{\phi}_{0}+\vec{R}(t,s).\nonumber
\end{eqnarray}
Furthermore, by the definition of $P_{b}(\mathscr{H}_{1},t)$ in Definition \ref{def-3}, we consider
\begin{eqnarray}
\label{eq-42}
\mathcal {M}_{\alpha_{1}, \gamma_{1}}(t)\vec{\phi}(t,s)&=&\mathcal {M}_{\alpha_{1}, \gamma_{1}}(t)\mathcal {U}(t,s)\vec{\phi}_{0}\nonumber\\&=&P_{b}(\mathscr{H}_{1})\mathcal {M}_{\alpha_{1}, \gamma_{1}}(t)\mathcal {U}(t,s)\vec{\phi}_{0}+\mathcal {M}_{\alpha_{1}, \gamma_{1}}(t)\vec{R}_{1}(t,s).
\end{eqnarray}
It is easy to see that $\mathcal {M}_{\alpha_{1}, \gamma_{1}}(t)\mathcal {U}(t,s)\vec{\phi}_{0}$
is the solution of the problem
\begin{eqnarray}\label{eq-44}
i\partial_{t}\vec{\phi}=\widetilde{\mathscr{H}}(t,s)\vec{\phi}, \ \ \ \ \ \ \vec{\phi}(s)=\mathcal {M}_{\alpha_{1}, \gamma_{1}}(s)\vec{\phi}_{0}
 \end{eqnarray}
 with  $\widetilde{\mathscr{H}}(t)$ defined by (\ref{eq-38}).
Now decompose
\begin{eqnarray}\label{eq-43}
P_{b}(\mathscr{H}_{1})\mathcal {M}_{\alpha_{1}, \gamma_{1}}(t)\mathcal {U}(t,s)\vec{\phi}_{0}=\sum^{M_{1}}_{j=1}A_{j}(t,s)e^{-i\omega_{j}(t-s)}\vec{\psi}_{j}+A_{0}(t,s)\vec{\psi}_{0},
\end{eqnarray}
for some unknown  functions $A_{j}$, where $\vec{\psi}_{j}\in L_{j}$ with $L_{j}$ defined as in Lemma \ref{lem-2} is the generalized eigenspace of $\mathscr{H}_{1}$. After substituting the decomposition (\ref{eq-42}) and (\ref{eq-43}) in (\ref{eq-44}) and notice that $P_{b}(\mathscr{H}_{1})\mathcal {M}_{\alpha_{1}, \gamma_{1}}(t)\vec{R}_{1}(t,s)=0$, we obtain the equation
\begin{eqnarray}\label{eq-45}
\sum^{M_{1}}_{j=1}\frac{1}{i}\partial_{t}A_{j}(t,s)e^{-i\omega_{j}(t-s)}\vec{\psi}_{j}+
\frac{1}{i}\partial_{t}A_{0}(t,s)\vec{\psi}_{0}+A_{0}(t,s)H_{1}\vec{\psi}_{0}&=&-P_{b}(\mathscr{H}_{1})\widetilde{V}_{2}(t,x-\vec{e}_{1}t)\widetilde{\vec{\phi}}(t,s)\nonumber\\
&:=&\sum^{M}_{j=0}\vec{g}_{j}(t,s)
\end{eqnarray}
where $\vec{g}_{j}(t,s)\in L_{j}$ is exponentially decaying in $t$ for any fixed $s$, $\widetilde{V}_{2}$ is defined in (\ref{eq-38})  and $\widetilde{\vec{\phi}}(t,s):=\mathcal {M}_{\alpha_{1}, \gamma_{1}}(t)\vec{\phi}(t,s)$. Hence we have
\begin{eqnarray}\label{eq-46}
&&\frac{1}{i}\partial_{t}A_{j}(t,s)e^{-i\omega_{j}(t-s)}\vec{\psi}_{j}=\vec{g}_{j}(t,s),\nonumber\\
&&\frac{1}{i}\partial_{t}A_{0}(t,s)\vec{\psi}_{0}+A_{0}(t,s)H_{1}\vec{\psi}_{0}=\vec{g}_{0}(t,s),
\end{eqnarray}
Since  $|\vec{g}_{j}(t,s)|\lesssim e^{-\alpha t}$, it follows that $A_{j}(t,s)$ has a limit $A_{j}(s)$ as $t\rightarrow \infty$ for $1\leq j\leq M_{1}$.  Moreover, applying $\mathscr{H}_{1}$ to (\ref{eq-46}) one obtains
$$\frac{1}{i}\partial_{t}A_{0}(t,s)\mathscr{H}_{1}\vec{\psi}_{0}=\mathscr{H}_{1}\vec{g}_{0}(t,s),$$
Since the right-hand side decays exponentially in $t$ and $\mathscr{H}_{1}\vec{\psi}_{0}$ is a localized function,
we can similarly obtain that $A_{0}(t,s)$ has a limit $A_{0}(s)$ as $t\rightarrow \infty$. Thus
\begin{eqnarray}\label{eq-47}
\Big\|P_{b}(\mathscr{H}_{1},t)\mathcal {U}(t,s)\vec{\phi}_{0}-\sum^{M_{1}}_{j=0}A_{j}(s)e^{-i\omega_{j}(t-s)}\mathcal {M}_{\alpha_{1}, \gamma_{1}}(t)^{-1}\vec{\psi}_{j}\Big\|_{L^{2}}\rightarrow 0, \ \ \ t\rightarrow +\infty.
\end{eqnarray}
Let $\widetilde{\mathcal {U}}(t,s)$ be defined by (\ref{commut-1}) and $\vec{u}_{j}(s)={M}_{\alpha_{1}, \gamma_{1}}(s)^{-1}\widetilde{\Omega}^{-}_{1}(s)\vec{\psi}_{j}$, notice that
\begin{eqnarray}\label{eq-48}
\widetilde{\mathcal {U}}(t,s)\widetilde{\mathcal {U}}(s,r)\mathcal{U}_{1}(r,s)\vec{\psi}_{j}-\mathcal{U}_{1}(t,s)\vec{\psi}_{j}
=\big(\widetilde{\mathcal {U}}(t,r)\mathcal{U}_{1}(r,t)-I\big)\mathcal{U}_{1}(t,s)\vec{\psi}_{j},\nonumber
\end{eqnarray}
and by the proof of Lemma \ref{thm-4},  $\widetilde{\mathcal {U}}(t,r)\mathcal{U}_{1}(r,t)P_{b}(\mathscr{H}_{1})-I\rightarrow 0$ as $r,t\rightarrow +\infty$  in $L^{2}$, we have
\begin{eqnarray}\label{eq-49}
\Big\|\mathcal {U}(t,s)\big(\sum^{M_{1}}_{j=0}A_{j}(s)\vec{u}_{j}(s)\big)-\sum^{M_{1}}_{j=0}A_{j}(s)e^{-i\omega_{j}(t-s)}\mathcal {M}_{\alpha_{1}, \gamma_{1}}(t)^{-1}\vec{\psi}_{j}\Big\|_{L^{2}}\rightarrow 0, \ \ \ t\rightarrow +\infty,
\end{eqnarray}
which combing (\ref{eq-47}) implies
\begin{eqnarray}\label{eq-50}
\Big\|\mathcal {U}(t,s)\big(\sum^{M_{1}}_{j=0}A_{j}(s)\vec{u}_{j}(s)\big)-P_{b}(\mathscr{H}_{1},t)\mathcal {U}(t,s)\vec{\phi}_{0}\Big\|_{L^{2}}\rightarrow 0, \ \ \ t\rightarrow +\infty.
\end{eqnarray}
Similarly, if we denote $K_{j}$ $(0\leq j\leq M_{2})$ the generalized eigenspaces for $\mathscr{H}_{2}$, it follows that as $ t\rightarrow +\infty$,
\begin{eqnarray}
\label{eq-47'}
\Big\|P_{b}(\mathscr{H}_{2},t)\mathcal {U}(t,s)\vec{\phi}_{0}-\sum^{M_{2}}_{j=0}B_{j}(s)e^{-i\mu_{j}(t-s)}\mathcal {G}_{-\vec{e}_{1}}(t)\mathcal {M}_{\alpha_{1}, \gamma_{1}}(t)^{-1}\vec{h}_{j}\Big\|_{L^{2}}\rightarrow 0,
\end{eqnarray}
and
\begin{eqnarray}
\label{eq-51'}
\Big\|\mathcal {U}(t,s)\big(\sum^{M_{2}}_{j=0}B_{j}(s)\vec{v}_{j}(s)\big)-\sum^{M_{2}}_{j=0}B_{j}(s)e^{-i\mu_{j}(t-s)}\mathcal {G}_{-\vec{e}_{1}}(t)\mathcal {M}_{\alpha_{1}, \gamma_{1}}(t)^{-1}\vec{h}_{j}\Big\|_{L^{2}}\rightarrow 0,
\end{eqnarray}
for $\vec{h}_{j}\in K_{j}$ and some functions $B_{j}$ of $s$.
Then we have
\begin{eqnarray}\label{eq-51}
\Big\|\mathcal {U}(t,s)\big(\sum^{M_{2}}_{j=0}B_{j}(s)\vec{v}_{j}(s)\big)-P_{b}(\mathscr{H}_{2},t)\mathcal {U}(t,s)\vec{\phi}_{0}\Big\|_{L^{2}}\rightarrow 0, \ \ t\rightarrow +\infty,
\end{eqnarray}
where $\vec{v}_{j}(s)=\mathcal {G}_{\vec{e}_{1}}(s)^{-1}{M}_{\alpha_{1}, \gamma_{1}}(s)^{-1}\widetilde{\Omega}^{-}_{2}(s)\vec{h}_{j}$ and
\begin{eqnarray}
\widetilde{\Omega}^{-}_{2}(s)=s-\lim_{t\rightarrow+\infty}\widehat{\mathcal{U}}(s,t)\mathcal{U}_{2}(t,s)P_{b}(\mathscr{H}_{2}).\nonumber
\end{eqnarray}
Here $\widehat{\mathcal{U}}(t)$ the propagator  of
$i\partial_{t}\vec{\psi}=\widehat{\mathscr{H}}(t)\vec{\psi}$ with
\begin{eqnarray}\label{eq-52}
\widehat{\mathscr{H}}(t)=\mathscr{H}_{2}+\widetilde{V}_{1}(t,x+\vec{e}_{1}t) \nonumber
\end{eqnarray}
where $\mathscr{H}_{2}$ is defined by (\ref{eq-34}),
\begin{eqnarray}
\widetilde{V}_{1}(t,x)=\left(
  \begin{array}{cc}
    U_{1}(x) & -e^{i(\theta_{1}-\theta_{2})(t,x)}W_{1}(x) \\
    e^{-i(\theta_{1}-\theta_{2})(t,x)}W_{1}(x) & -U_{1}(x) \\
  \end{array}
  \right), \nonumber
\end{eqnarray}
and $\theta_{1}, \theta_{2}$ are defined by (\ref{eq-35}).

Now we make decomposition
\begin{eqnarray}\label{eq-53}
\vec{\phi}_{0}=\sum^{M_{1}}_{j=0}A_{j}(s)\vec{u}_{j}(s)+\sum^{M_{2}}_{j=0}B_{j}(s)\vec{v}_{j}(s)+\vec{f}(s)
\end{eqnarray}
where $\vec{f}(s):=\vec{\phi}_{0}-\sum^{M_{1}}_{j=0}A_{j}(s)\vec{u}_{j}(s)-\sum^{M_{2}}_{j=0}B_{j}(s)\vec{v}_{j}(s)$.
It follows
\begin{eqnarray}
P_{b}(\mathscr{H}_{1})\mathcal {U}(t,s)\vec{f}(s)&=&P_{b}(\mathscr{H}_{1})\mathcal {U}(t,s)\vec{\phi}_{0}\nonumber\\
&&\ \ \ \ \ -P_{b}(\mathscr{H}_{1})\mathcal {U}(t,s)\big(\sum^{M_{1}}_{j=0}A_{j}(s)\vec{u}_{j}(s)\big)-P_{b}(\mathscr{H}_{1})\mathcal {U}(t,s)\big(\sum^{M_{2}}_{j=0}B_{j}(s)\vec{v}_{j}(s)\big).\nonumber
\end{eqnarray}
By using (\ref{eq-50}) and the identity $P^{2}_{b}(\mathscr{H}_{1})=P_{b}(\mathscr{H}_{1})$, we obtain
\begin{eqnarray}\label{eq-53''}
\Big\|P_{b}(\mathscr{H}_{1})\mathcal {U}(t,s)\vec{\phi}_{0}-P_{b}(\mathscr{H}_{1})\mathcal {U}(t,s)\big(\sum^{M_{1}}_{j=0}A_{j}(s)\vec{u}_{j}(s)\big)\Big\|_{L^{2}}\rightarrow 0, \ \ \ t\rightarrow \infty.
\end{eqnarray}
Furthermore,  $P_{b}(\mathscr{H}_{1})\sum^{M_{2}}_{j=0}B_{j}(s)e^{-i\mu_{j}(t-s)}\mathcal {G}_{-\vec{e}_{1}}(t)\mathcal {M}_{\alpha_{1}, \gamma_{1}}(t)^{-1}\vec{h}_{j}$  goes to zero in the $L^{2}$ sense as $t\rightarrow\infty$ due to the fact that $\mathcal {G}_{-\vec{e}_{1}}(t)\mathcal {M}_{\alpha_{1}, \gamma_{1}}(t)^{-1}\vec{h}_{j}$ actually has a moving center with position $\vec{e}_{1}t$ and  $P_{b}(\mathscr{H}_{1})$ is the projection onto the space of localized function, which combining (\ref{eq-51'}) imply
\begin{eqnarray}\label{eq-53'}
P_{b}(\mathscr{H}_{1})\mathcal {U}(t,s)\Big(\sum^{M_{2}}_{j=0}B_{j}(s)\vec{v}_{j}(s)\Big)\rightarrow 0, \ \ t\rightarrow\infty.
\end{eqnarray}
Thus by (\ref{eq-53''}) and (\ref{eq-53'}), it follows that $\vec{f}(s)\in \mathscr {A}(s)$.

Finally, we prove (\ref{eq-53}) is a direct sum decomposition. Here the modulation transform will not be taken into account for simplicity. Let $0\neq\vec{v}(s)\in \mathscr {A}(s)\cap {\rm Ran \Omega^{-}_{1}(s)}$, there exists $\vec{u}$ such that  $\vec{v}(s)=\Omega^{-}_{1}(s)\vec{u}$, without loss of generality, we assume $\vec{u}$ is a generalized eigenfunction of $\mathscr{H}_{1}$. Then similar to the proof of (\ref{eq-49}), we have
\begin{eqnarray}\label{eq-54}
\big\|P_{b}(\mathscr{H}_{1})\mathcal {U}(t,s)\vec{v}(s)-P_{b}(\mathscr{H}_{1})\mathcal {U}_{1}(t,s)\vec{u}\big\|_{L^{2}}\rightarrow 0, \ \ \ t\rightarrow +\infty, \nonumber
\end{eqnarray}
which contradict to the fact that $\vec{v}(s)\in \mathscr {A}(s)$. Let $0\neq\vec{v}(s)\in {\rm Ran \Omega^{-}_{2}(s)}\cap {\rm Ran \Omega^{-}_{1}(s)}$, there exist $\vec{u}_{j}$, $j=1,2$ such that
$$\vec{v}(s)=\Omega^{-}_{1}(s)\vec{u}_{1}=\Omega^{-}_{2}(s)\vec{u}_{2},$$
 applying the same argument of the proof of (\ref{eq-49}) again, we  obtain
$$\lim_{t\rightarrow +\infty}\big(P_{b}(\mathscr{H}_{1})\mathcal {U}_{1}(t,s)\vec{u}_{1}-\mathcal {G}_{-\vec{e_{1}}}(t)P_{b}(\mathscr{H}_{2})\mathcal {U}_{2}(t,s)\vec{u}_{2}\mathcal {G}_{\vec{e_{1}}}(s)\big)=0$$
in the sense of $L^{2}$, which combing the exponentially decaying of
$\vec{u}_{j}$, $j=1,2$ to get a contradiction. Hence we finish the proof.
\end{proof}

\begin{rem}{\rm
One may further seek for the orthogonal sum decomposition of form (\ref{ac-2}). In fact, it is ture if we use $\mathcal {U}^{\ast}(t,s)$, the conjugate operator of $\mathcal {U}(t,s)$ instead of $\mathcal {U}(s,t)$ in the definitions of wave operators $\Omega^{-}_{1}(s)$ and $\Omega^{-}_{2}(s)$.}
\end{rem}

%Denote by $P_{\kappa b}(s)$ the projection onto ${\rm Ran} \Omega^{-}_{\kappa}(s)$ ($\kappa=1,2$),

\begin{cor}
Let $s,t\in \mathbb{R}$ and $P_{c}(t)$ be the projection onto $\mathscr{A}(t)$, then we have
\begin{eqnarray}\label{commute-2}
\mathcal {U}(s,t)P_{c}(t)=P_{c}(s)\mathcal {U}(s,t).\nonumber
\end{eqnarray}
\end{cor}

\begin{proof}
The proof is exactly the same as the scalar case (see Lemma \ref{ex-le-1}), we omit the detail here.
\end{proof}

\subsection{The Strichartz estimates for matrix charge transfer model}

The matrix charge transfer model is not self-adjoint, one could not apply the $TT^{\ast}$ to the endpoint Strichartz  estimates even though the decay estimates hold. Here, by using asymptotic completeness and the properties of wave operators, we would show that whenever the initial data belongs to the ``scattering states", the Strichartz estimates hold for all admissible pair defined as in (\ref{pair}).

\begin{thm}\label{thm-6}
Consider the matrix charge transfer model (\ref{main-equation-2}) with $\mathscr{H}_{\kappa} \ (\kappa=1,2)$ being admissible in the sense of Definition \ref{defi-4} and satisfiying the stability condition (\ref{eq-32}). Let $\mathcal {U}(t)$ be its propagator satisfying the growth assumption (\ref{ap}). The for any initial date $\vec{\psi}_{0}\in \mathscr {A}(0)$ and admissible pair $(p,q)$ satisfying (\ref{pair}), one has the Strichartz estimates
\begin{eqnarray}\label{eq-55}
\big\|\mathcal {U}(t)\vec{\psi}_{0}\big\|_{L_{t}^{p}L_{x}^{q}}\lesssim \|\vec{\psi}_{0}\|_{L^{2}}.
\end{eqnarray}
\end{thm}

\begin{proof}
As in the scalar case, (\ref{eq-55}) is proved by  three steps, which could be summed as
$$\rm {Kato{-}Jensen  \Rightarrow Local\ decay\Rightarrow\ Strichartz\ estimates}.$$
The is basically identical with the scalar case, we would not write down these details.
\end{proof}

Finally, we will consider matrix charge transfer model with nonlinear term
\begin{eqnarray}\label{main-equation-6}
\left\{%
\begin{array}{ll}\label{main-equation-2}
i\partial_{t}\vec{\psi}=\mathscr{H}_{0}\vec{\psi}+V_{1}(t,x)\vec{\psi}+V_{2}(t,x-\vec{e}_{1}t)+\vec{\psi}+\vec{F}(t),\\
\vec{\psi}(0,\cdot)=\vec{\psi}_{0}\in \mathscr{A}(0),\ \ x\in \mathbb{R}^{n},\\
\end{array}
\right.
\end{eqnarray}
We project both side of the equation onto ``scattering states" just as in the scalar case, and then could obtain the following result.

\begin{thm}\label{thm-7}
Consider the nonlinear matrix charge transfer model (\ref{main-equation-6}) with $\mathscr{H}_{\kappa} \ (\kappa=1,2)$ being admissible in the sense of Definition \ref{defi-4} and satisfiying the stability condition (\ref{eq-32}). Let $\mathcal {U}(t)$ be the propagator of the equation (\ref{main-equation-2}) satisfying the growth assumption (\ref{ap}). Then for initial data $\vec{\psi}_{0}\in \mathscr{A}(0)$ and admissible pairs $(p,q),\ (\tilde{p},\tilde{q})$  satisfying (\ref{pair}), the solution $\vec{\psi}(t,x)$ satisfies the  Strichartz estimates,
\begin{eqnarray}\label{NHS}
\|P_{c}(t)\vec{\psi}(t)\|_{L^{p}_{t}L^{q}_{x}}\lesssim \|\vec{\psi}_{0}\|_{L^{2}}+\|\vec{F}\|_{L^{\tilde{p}'}_{t}L^{\tilde{q}'}_{x}}.
\end{eqnarray}
\end{thm}

\begin{proof}
We will still apply the ``five steps argument" used for scalar case, we here only briefly recall the steps of this argument. The main idea is as follows:
\begin{eqnarray}
{\rm Kato{-}Jense\ for\ linear\ propagator\ \mathcal{U}(t)} &\Rightarrow& {\rm Local\ decay\ for\ linear\ propagator\ \mathcal{U}(t)}\nonumber\\
&\Rightarrow& {\rm Local\ decay\ for\ source\ term}\nonumber\\
&\Rightarrow& {\rm Local\ decay\ for\ solution}\ \vec{\psi}\nonumber\\
&\Rightarrow& {\rm The\ Strichartz\ estimates.}\nonumber
\end{eqnarray}
We would omit the detail since the whole proof are essentially the same as the scalar one.
\end{proof}


\begin{thebibliography}{99}

\bibitem{AS}W. K. Abou Salem, \emph{Solitary wave dynamics in time-dependent potentials,} J. Math. Phys. 49 (2008), 032101.

\bibitem{AFS}W. K. Abou Salem, J. Fr\"{o}hlich and I. Sigal, \emph{Colliding solitons for the non-linear Schr\"{o}dinger equation,} Comm. Math. Phys. 291 (2009), 151-176.

\bibitem{Bam1}D. Bambusi and A. Maspero, \emph{Freezing of energy of a soliton in an external potential}, arXiv:1503.08608.

\bibitem{Bam2}D. Bambusi, \emph{Asymptotic stability of ground states in some Hamiltonian PDEs with symmetry,} Comm. Math. Phys.  320 (2013), 499-542.

\bibitem{Bec2}M. Beceanu, \emph{A centre-stable manifold for the focusing cubic NLS in $\mathbb{R}^{3+1}$},  Comm. Math. Phys.  280 (2008), 145-205.

\bibitem{Bec3}M. Beceanu, \emph{New estimates for a time-dependent Schr\"{o}dinger equation,} Duke Math. J. 159 (2011), 417-477.

\bibitem{BeSo1}M. Beceanu and A. Soffer, \emph{The Schr\"{o}dinger equation with potential in rough motion}, Comm. Partial Diff. Equ.  37 (2012), 969-1000.

\bibitem{BeSo2}M. Beceanu and A. Soffer, \emph{The Schr\"{o}dinger equation with potential in random motion}, arXiv:1111.4584.

\bibitem{BJ}J. Bronski and R. Jerrard, \emph{Soliton dynamics in a potential,} Math. Res. Lett. 7 (2000), 329-342.

\bibitem{BM} J. Bouclet and H. Mizutani, \emph{Uniform resolvent and Strichartz estimates for Schr\"odinger equations with
critical singularities}, arxiv: 1607.01187.

\bibitem{Bou}J. Bourgain, \emph{On long-time behaviour of solutions of linear Schr\"{o}dinger equations with smooth time-dependent potential}, Geometric aspects of functional analysis, Lecture Notes in Math., 1807, Springer, Berlin, 2003, 99-113.


\bibitem{Cai}K. Cai, \emph{Fine properties of charge transfer models}, arXiv: math/0311048v1.


\bibitem{Cuc1}S. Cuccagna, \emph{Stabilization of solutions to nonlinear Schr\"{o}dinger equations}, Comm. Pure Appl. Math. 54 (2001), 1110-1145.

\bibitem{CucMa1}S. Cuccagna and M. Maeda, \emph{On weak interaction between a ground state and a non-trapping potential}, J. Differ. Equ. 256 (2014), 1395-1466.

\bibitem{CuMa2}S. Cuccagna and M. Maeda, \emph{On weak interaction between a ground
state and a trapping potential,} Discrete Contin. Dyn. Syst. 35 (2015), 3343-3376.

\bibitem{CuMi}S. Cuccagna and T. Mizumachi, \emph{On asymptotic stability in energy space of ground states for Nonlinear Schr\"{o}dinger equations}, Comm. Math. Phys. 284 (2008), 51-77.

%\bibitem{DH}K. Datchev and J. Holmer, \emph{Fast soliton scattering by attractive delta impurities,}
%Comm. Partial Diff. Equ. 34 (2009), 1074-1113.

\bibitem{ES}B. Erdo$\hat{g}$an and W. Schlag, \emph{Dispersive estimates for Schr\"{o}dinger operators in the presence of a resonance and/or an eigenvalue at zero energy in dimension three: II,} Journal d'Analyse Math\'{e}matique, 99 (2006), 199-248.

%\bibitem{Fos}D. Foschi, \emph{Inhomogeneous Strichartz estimates,} J. Hyperbolic Differ. Equ.  2 (2005), 1-24.

\bibitem{FJGS}J. Fr\"{o}hlich, B. Jonsson, S. Gustafson and I. Sigal, \emph{Solitary wave dynamics in an external potential}, Comm. Math. Phys. 250 (2004), 613-642.

\bibitem{GV}J. Ginibre, G. Velo, \emph{Generalized Strichartz Inequalities for the Wave Equation}, J. Func.
Anal., 133 (1995), 50-68.
\bibitem{Go}M. Goldberg, \emph{Strichartz estimates for the Schr\"odinger equation with time periodic
$L^{n/2}$ potentials}, Journal of Functional Analysis,  256(2009), 718¨C746.

\bibitem{Gr}J. Graf, \emph{Phase Space Analysis of the Charge transfer Model}, Helv. Physica Acta 63 (1990), 107-138.


\bibitem{GNP}S. Gustafson, K. Nakanishi and T. Tsai, \emph{Asymptotic stability and completeness in the energy space for nonlinear Schr\"{o}dinger equations with small solitary waves,} Int. Math. Res. Not. 66 (2004), 3559-3584.

\bibitem{HM}J. Holmer and  M. Zworski, \emph{Soliton interaction with slowly varying potentials,} Internat. Math. Res. Notices  (2008), Art. ID runn026, 36 pp.


\bibitem{HL}D. Hundertmark and Y. Lee, \emph{Exponential decay of eigenfunctions and generalized eigenfunction of non self-adjoint matrix Schr\"{o}dinger operators related to NLS}, Bulletin of the London Mathematical Society, 39 (2007), 709-720.

\bibitem{JK}A. Jensen and T. Kato, \emph{Spectral properties of Schr\"{o}dinger operators and time-decay of the wave functions,} Duke Math. J.  46 (1979), 583-611.



\bibitem{JSS2}J. Journ\'{e},  A. Soffer and C. Sogge, \emph{$L^{p}\rightarrow L^{p'}$ Estimates for time-dependent Schr\"{o}dinger operators}, Bull. Amer. Math. Soc. (N.S.) 23 (1990), 519-524.

\bibitem{JSS1}J. Journ\'{e},  A. Soffer and C. Sogge, \emph{Decay estimates for Schr\"{o}dinger operators,} Comm. Pure Appl. Math. 44 (1991), 573-604.

\bibitem{Kato1}T. Kato, \emph{On the Adiabatic Theorem of Quantum Mechanics,} J. Phys. Soc. Japan (1950), 435-439.

\bibitem{KT}M. Keel and T. Tao,  \emph{Endpoint Strichartz Estimates}, Amer. J. Math. 120 (1998), 955-980.

%\bibitem{KTV}H. Koch, D. Tataru and M. Visan, \emph{Dispersive Equations and Nonlinear Waves}, Oberwolfach Seminars, Birkh\"{a}user (2014).

\bibitem{MM1} Y. Martel and F. Merle, \emph{Multi solitary waves for nonlinear Schr\"{o}dinger equations,} Ann. Inst. H. Poincar\'{e}. Anal. Non Lin. 23 (2006), 849-864.

\bibitem{MMT} Y. Martel, F. Merle and T. Tsai, \emph{Stability in $H^{1}$ of the sum of K solitary waves for some nonlinear Schro\"{o}dinger equations,} Duke Math. J. 133 (2006), 405-466.


\bibitem{Per2}G. Perelman, \emph{Some results on the scattering of weakly interacting solitons for nonlinear Schr\"{o}dinger equation.} In: Demuth et al., M., eds. Spectral Theory, Microlocal Analysis, Singular Manifolds. Math. Top. 14. Berlin: Akademie Verlag, (1997), 78-137.

%\bibitem{Per3}G. Perelman, \emph{On the formation of singularities in solutions of the critical nonlinear Schr\"{o}dinger equation,} Ann. Henri Poincar\'{e} 2 (2001), 605-673.

\bibitem{Per1}G. Perelman, \emph{Asymptotic stability of multi-soliton solutions for nonlinear Schr\"{o}dinger equations}, Comm. Partial Diff. Equ. 29 (2004), 1051-1095.

\bibitem{Per5}G. Perelman, \emph{A remark on soliton-potential interactions for nonlinear Schr\"{o}dinger equations,}  Math. Res. Lett. 16 (2009), 477-486.

\bibitem{Per4} G. Perelman, \emph{Two soliton collision for nonlinear Schr\"{o}dinger equations in dimension 1}, Ann. Inst. H. Poincar\'{e}. Anal. Non Lin. 28 (2011), 357-384.

\bibitem{RS1} M. Reed, B. Simon, \emph{ Methods of Modern Mathematical Physics}, Vol. II: Fourier Analysis, Self-Adjointness, Academic Press, New York 1975.

\bibitem{RS}M. Reed, B. Simon, \emph{ Methods of Modern Mathematical Physics}, Vol. IV: Analysis of Operators, Academic Press, New York 1978.

\bibitem{RoSc}I. Rodnianski and W. Schlag, \emph{Time decay for solutions of Schr\"{o}dinger equations with rough and time-dependent potentials}, Invent. Math. 155 (2004), 451-513.

\bibitem{RSS1}I. Rodnianski, W. Schlag and A. Soffer, \emph{Dispersive analysis of charge transfer models,} Comm. Pure Appl. Math. 58 (2005), 149-216.

\bibitem{RSS2}I. Rodnianski, W. Schlag and A. Soffer, \emph{Asymptotic stability of N-soliton states of NLS,}  arXiv:math/03091114v1.

\bibitem{Sch}W. Schlag, \emph{Stable Manifolds for an orbitally unstable NLS,} Ann. Math. 169 (2009) 139-227.

\bibitem{Sch2}W. Schlag, \emph{Dispersive estimates for Schr\"{o}dinger operators: a survey,} Mathematical aspects of nonlinear dispersive equations, 255-285, Ann. of Math. Stud., 163, Princeton Univ. Press, Princeton, NJ, 2007.

\bibitem{Wei}M. Weinstein, \emph{Localized States and Dynamics in the Nonlinear Schr\"{o}dinger/Gross-Pitaevskii Equation},  to appear in Frontiers in Applied Dynamics: Reviews and Tutorials.

\bibitem{Wein2}M. Weinstein, \emph{Modulational stability of ground states of nonlinear Schr\"{o}dinger equations}, SIAM J. Math. Anal. 16 (1985), 472-491.

\bibitem{Wu}U. W\"{u}ller, \emph{Geometric Methods in Scattering Theory of the Charge Transfer Model}, Duke Math J.  62 (1991), 273-313.

\bibitem{Ya1}K, Yajima, \emph{A multichannel scattering theory for some time dependent Hamiltonians, charge transfer problem,} Comm. Math. Phys. 75 (1980),153-178.

\bibitem{Ya2} K. Yajima, \emph{The $W^{k, p}$-continuity of wave operators for Schr\"odinger operators}, Journal of the Mathematical Society of Japan  (3)47, 551-581, 1995.

\bibitem{Zi}L. Zielinski, \emph{Asymptotic completeness for multiparticle dispersive charge transfer models}, J. Funct. Anal. 150 (1997), 453-470.

\bibitem{SZ}G. Zhou and I. Sigal, \emph{Relaxation of solitons in Nonlinear Schr\"{o}dinger equations with potential,} Adv Math.  216 (2007), 443-490.




\end{thebibliography}
  \end{document}